\documentclass[final,notitlepage,12pt,reqno,tbtags]{amsart}
\usepackage{graphicx}
\usepackage{calc}
\usepackage{url}
\usepackage{mathrsfs}
\usepackage{rotating}
\newlength{\depthofsumsign}
\makeatletter
\let\I\@undefined
\makeatother
\usepackage{geometry}
\usepackage{indentfirst}
\usepackage[normalem]{ulem}
\usepackage{float}
\usepackage{amsthm}

\usepackage{enumitem}[2011/09/28]
\setenumerate{align=left, leftmargin=0pt,labelsep=.5em, labelindent=0\parindent,listparindent=\parindent,itemindent=*}
\usepackage{array}
\usepackage{empheq}
\usepackage{natbib}
\usepackage[all]{xy}

\usepackage{multirow}

\newbox\shell
\newcommand{\dia}[2]{\setbox\shell=\hbox{\begin{picture}(180,120)(-90,-60)#1
\put(-90,-60){\makebox(180,120)[b]{\large #2}}\end{picture}}\dimen0=\ht
\shell\multiply\dimen0by7\divide\dimen0by16\raise-\dimen0\box\shell\hfill}

\newcommand{\vtx}{\circle*{10}}

\usepackage{caption}[2012/02/19]

\usepackage{longtable,lscape}

\usepackage{mathptmx}       
\DeclareSymbolFont{operators}{OT1}{txr}{m}{n}
\SetSymbolFont{operators}{bold}{OT1}{txr}{bx}{n}
\def\operator@font{\mathgroup\symoperators}
\DeclareSymbolFont{italic}{OT1}{txr}{m}{it}
\SetSymbolFont{italic}{bold}{OT1}{txr}{bx}{it}
\DeclareSymbolFontAlphabet{\mathrm}{operators}
\DeclareMathAlphabet{\mathbf}{OT1}{txr}{bx}{n}
\DeclareMathAlphabet{\mathit}{OT1}{txr}{m}{it}
\SetMathAlphabet{\mathit}{bold}{OT1}{txr}{bx}{it}
\DeclareSymbolFont{letters}{OML}{txmi}{m}{it}
\SetSymbolFont{letters}{bold}{OML}{txmi}{bx}{it}
\DeclareFontSubstitution{OML}{txmi}{m}{it}
\DeclareSymbolFont{lettersA}{U}{txmia}{m}{it}
\SetSymbolFont{lettersA}{bold}{U}{txmia}{bx}{it}
\DeclareFontSubstitution{U}{txmia}{m}{it}
\DeclareSymbolFontAlphabet{\mathfrak}{lettersA}
\DeclareSymbolFont{symbols}{OMS}{txsy}{m}{n}
\SetSymbolFont{symbols}{bold}{OMS}{txsy}{bx}{n}
\DeclareFontSubstitution{OMS}{txsy}{m}{n}

\usepackage{amssymb,bm,amsmath}
\usepackage{amsfonts}

\usepackage[OT2,T1,T2A]{fontenc}

\usepackage[english,french,german,russian]{babel}
\usepackage{appendix}

\DeclareMathOperator{\IKM}{\mathbf{IKM}}
\DeclareMathOperator{\IvKM}{\mathbf{\widetilde IKM}}

\DeclareMathOperator{\IKvM}{\mathbf{I\widetilde KM}}

\DeclareMathOperator{\IpKM}{\mathbf{\acute IKM}}

\DeclareMathOperator{\IKpM}{\mathbf{I\acute KM}}

\DeclareMathOperator{\D}{d}
\DeclareMathOperator{\I}{Im}
\DeclareMathOperator{\R}{Re}

\def\XXint#1#2#3{{\setbox0=\hbox{$#1{#2#3}{\int}$}
     \vcenter{\hbox{$#2#3$}}\kern-.5\wd0}}

\bibpunct{[}{]}{,}{n}{}{;}

\def\eor{\hfill$ \square$}

\theoremstyle{plain}
\newtheorem{theorem}{Theorem}[section]
\newtheorem{proposition}[theorem]{Proposition}
\newtheorem{lemma}[theorem]{Lemma}

\newtheorem{conjecture}[theorem]{Conjecture}

\newenvironment{remark}[1][Remark]{\begin{trivlist}
\item[\hskip \labelsep {\bfseries #1}]}{\end{trivlist}}

\theoremstyle{definition}
\newtheorem{definition}[theorem]{Definition}

\numberwithin{equation}{section}

\usepackage{color}

\begin{document}

\pagenumbering{roman}
\selectlanguage{english}
\title[Broadhurst--Mellit determinants]{Wro\'nskian factorizations and \\Broadhurst--Mellit determinant formulae}
\author[Yajun Zhou]{Yajun Zhou}
\address{Program in Applied and Computational Mathematics (PACM), Princeton University, Princeton, NJ 08544} \email{yajunz@math.princeton.edu}\curraddr{ \textsc{Academy of Advanced Interdisciplinary Studies (AAIS), Peking University, Beijing 100871, P. R. China}}\email{yajun.zhou.1982@pku.edu.cn}

\date{\today}
\thanks{\indent\textit{Subject Classification (AMS 2010)}: 33C10, 46E25, 15A15, 11R06 (Primary) 81T18, 81T40, 81Q30, 60G50 (Secondary)\\\indent* This research was supported in part  by the Applied Mathematics Program within the Department of Energy
(DOE) Office of Advanced Scientific Computing Research (ASCR) as part of the Collaboratory on
Mathematics for Mesoscopic Modeling of Materials (CM4).
}
\maketitle

\begin{abstract}
     Drawing on Vanhove's contributions to mixed Hodge structures  for Feynman integrals in two-di\-men\-sion\-al quantum field theory, we compute two families of determinants whose entries are Bessel moments. Via explicit factorizations of certain Wro\'nskian determinants, we verify   two recent conjectures proposed by Broadhurst and Mellit, concerning determinants of arbitrary sizes. With some extensions to our methods, we also relate two more determinants of Broadhurst--Mellit to the logarithmic Mahler measures of certain polynomials. \end{abstract}

\tableofcontents

\clearpage

\pagenumbering{arabic}

\section{Introduction}
In perturbative expansions for two-dimensional quantum field theory, we often need to evaluate Feynman diagrams such as \cite[][\S8]{Vanhove2014Survey}\begin{align}
\;\;\;\;\;
\dia{\put(-50,0){\line(-1,0){50}}
\put(50,0){\line(1,0){50}}
\put(-100,0){$^M$}
\put(0,0){\circle{100}}
\qbezier(-50, 0)(0, 50)(50, 0)
\qbezier(-50, 0)(0, -50)(50, 0)
\put(50,0){\vtx}
\put(-50,0){\vtx}
\put(80,0){$^M$}\put(-15,45){$^{m_1}$}
\put(-4,-14){$ \vdots$}
\put(-15,20){$^{m_2 }$}
\put(-25,-47){$^{m_{n-1} }$}
\put(-15,-75){$^{m_n }$}
}{}\;\;\;=2^{n-1}\int_0^\infty I_0(Mx)\left[ \prod_{i=1}^n K_0(m_ix) \right]x\D x,
\end{align}where $ I_0(t)=\frac{1}{\pi}\int_0^\pi e^{t\cos\theta}\D\theta$ and $K_0(t)=\int_0^\infty e^{-t\cosh u}\D u $ are modified Bessel functions of zeroth order. When all the external legs and all the internal lines bear the same parameters (say, $ M=m_1=\cdots=m_n=1$ in the diagram above), we are left with the  single-scale Bessel moments  \cite{Groote2007,BBBG2008,Broadhurst2016,BroadhurstMellit2016}  \begin{align} \IKM(a,b;n):=\int_0^\infty[I_0(t)]^a[K_0(t)]^{b}t^{n}\D t\end{align}for certain non-negative integers $a,b,n\in\mathbb Z_{\geq0}$.

In addition to their important r\^oles in the computation of anomalous magnetic dipole moment \cite{LaportaRemiddi1996,Laporta2008,Laporta:2017okg} in quantum electrodynamics, these single-scale Bessel moments are also   intimately related  to motivic integrations  in algebraic geometry  \cite{BlochKerrVanhove2015}  and modular forms in number theory  \cite{Samart2016}, thus   having stimulated intensive mathematical research. For example, various linear relations among Bessel moments, such as $ \pi^2\IKM(3,5;1)=\IKM(1,7;1)$
\cite[conjectured in][(148)]{Broadhurst2016} and      $ 9\pi^2\IKM(4,4;1)=14\IKM(2,6;1)$   \cite[conjectured in][(147)]{Broadhurst2016}
 had been discovered by numerical experiments, before their formal proofs \cite{HB1,Zhou2017WEF} were constructed by algebraic and analytic methods.

Recently, based on  a collaboration with Anton Mellit \cite{BroadhurstMellit2016}, David Broadhurst  has laid out several dazzling conjectures about non-linear algebraic relations among $ \IKM(a,b;n)$ with fixed $a+b$  and varying $n$  \cite{Broadhurst2016}. They revolve around certain determinants whose entries are Bessel moments, two  of which are  recapitulated below.
\begin{conjecture}[Broadhurst--Mellit {\cite[][Conjecture 4]{Broadhurst2016}}]\label{conj:BMdetM}If $\mathbf M_k$  is a $k\times k$
matrix with elements
\begin{equation}
(\mathbf M_k)_{a,b}:=\int_0^\infty[I_0(t)]^a[K_0(t)]^{2k+1-a}t^{2b-1}\D t,
\label{Mk}\end{equation}then its determinant evaluates to\begin{align}
\det\mathbf M_k=\prod_{j=1}^k\frac{(2j)^{k-j}\pi^j}{\sqrt{(2j+1)^{2j+1}}}.
\end{align}
\end{conjecture}

\begin{conjecture}[Broadhurst--Mellit {\cite[][Conjecture 7]{Broadhurst2016}}]\label{conj:BMdetN}If $\mathbf N_k$  is a $k\times k$
matrix with elements
\begin{equation}
(\mathbf N_k)_{a,b}:=\int_0^\infty[I_0(t)]^a[K_0(t)]^{2k+2-a}t^{2b-1}\D t,
\label{Nk}\end{equation}then its determinant evaluates to\begin{align}
\det\mathbf N_k=\frac{2\pi^{(k+1)^2/2}}{\Gamma((k+1)/2)}\prod_{j=1}^{k+1}\frac{(2j-1)^{k+1-j}}{(2j)^j},
\end{align}an expression that involves Euler's gamma function  $ \Gamma(x):=\int_0^\infty t^{x-1}e^{-t}\D t$ for $x>0$.
\end{conjecture}
In our previous work \cite[][\S3]{Zhou2017WEF}, we established the determinant formula \begin{align} \det \mathbf M_2=\det
\begin{pmatrix}\IKM(1,4;1) & \IKM(1,4;3) \\
\IKM(2,3;1) & \IKM(2,3;3) \\
\end{pmatrix}=\frac{2\pi^{3}}{\sqrt{3^35^5}}\end{align} by evaluating all the four entries of $ \mathbf M_2$ in closed form. These analytic evaluations were made possible by integrations of some special modular forms. It appears uneconomical, if not utterly infeasible, to probe into the remaining scenarios in  Conjectures \ref{conj:BMdetM} and \ref{conj:BMdetN}  through  analytic expressions for all the individual elements in these matrices. Indeed, only a limited number of individual Bessel moments $ \IKM(a,b;n)$ for $a+b\geq5$ are currently known in closed form (say, as special $L$-values attached to certain automorphic forms) \cite{Broadhurst2016,Zhou2017WEF}.

 In this work, we  verify Conjectures \ref{conj:BMdetM}--\ref{conj:BMdetN}, in their entirety, via  Vanhove's studies of mixed Hodge structures for Feynman integrals \cite{Vanhove2014Survey},
and factorizations of certain Wro\'nskian determinants. This approach allows us to find a recursive mechanism underlying the Broad\-hurst--Mellit determinant formulae, without going through the ordeals of evaluating individual matrix elements by brute force. The same method can be extended to certain determinants whose entries involve the vacuum diagrams $V_n:=\IKM(0,n;1)=\int_0^\infty[K_0(t)] ^nt\D t$ for $n\in\{5,6\}$. These extensions allow us to evaluate two other determinants that were studied numerically by Broadhurst--Mellit \cite[][(101) and (114)]{Broadhurst2016}, in terms of logarithmic Mahler measures,  which are defined as\begin{align}
m(P):=\int_0^1\D t_1\cdots\int_0^1\D t_n\,\log|P(e^{2\pi i t_1},\dots,e^{2\pi i t_n})|\label{eq:defn_Mahler_m}
\end{align}for all Laurent polynomials $ P\in\mathbb C[x_1^{\pm1},\dots, x_n^{\pm1}]$.

This article runs as follows.  In \S\ref{sec:detM2}, we write a new proof for  $ \det \mathbf M_2=\frac{2\pi^{3}}{\sqrt{3^35^5}}$, using algebraic manipulations of determinants, rather than automorphic representations of individual matrix entries.  We carry on these algebraic arguments in  \S\ref{sec:detN2} to produce a proof of $ \det \mathbf N_2=\frac{\pi^{4}}{2^{6}3^2}$, before devoting  \S\ref{sec:detMkNk}  to the treatments of $ \det \mathbf M_k$ and $ \det\mathbf N_k$ that come in arbitrary sizes ($k\in\mathbb Z_{\geq2} $). In \S\ref{sec:VacMahler}, we open with an overview of current understandings for the relations between vacuum diagrams and Mahler measures, before presenting a proof of the  results stated below.\begin{theorem}[Broadhurst--Mellit determinants and Mahler measures]\label{thm:BM_Mahler}We have the following determinant evaluations, in terms of the logarithmic Mahler measures defined in \eqref{eq:defn_Mahler_m}:\begin{align}
\det \check {\mathbf M}_2:=\det
\begin{pmatrix}\IKM(0,5;1) & \IKM(0,5;3) \\
\IKM(2,3;1) & \IKM(2,3;3) \\
\end{pmatrix}={}&\frac{2 \pi^{3}}{15\sqrt{15}}m(1+x_1+x_2+x_3+x_4),\\\det \check {\mathbf N}_2:=\det
\begin{pmatrix}\IKM(0,6;1) & \IKM(0,6;3) \\
\IKM(2,4;1) & \IKM(2,4;3) \\
\end{pmatrix}={}&\frac{\pi^{4}}{96}m(1+x_1+x_2+x_3+x_4+x_{5})
.
\end{align}\end{theorem}
\section{An algebraic evaluation of $ \det \mathbf M_2$\label{sec:detM2}}
As announced in the introduction, we now calculate  $ \det \mathbf M_2$ without evaluating each element in the matrix $\mathbf M_2$. In \S\ref{subsec:W3x3}, using variations on the single-scale Bessel moments, we  construct a $3\times 3$ Wro\'nskian determinant as a function $\varOmega_3(u)$  of a  parameter $u\in(0,4)$, and characterize $\varOmega_3(u),u\in(0,4)$ up to an overall multiplicative constant. In \S\ref{subsec:detM2}, we determine the aforementioned multiplicative constant by the asymptotic behavior $ \varOmega_3(u),u\to0^+$, and compute $\det\mathbf M_2$ via the special value $ \varOmega_3(1)$.   \subsection{A $3\times 3$ Wro\'nskian determinant\label{subsec:W3x3}}To simplify notations, we introduce a few abbreviations involving Bessel moments and their analogs.  \begin{definition}We write $ \IvKM$ (resp. $\IKvM$) for two-scale Bessel moments with  a rescaled argument in one $I_0$ (resp.~$K_0$) factor. Concretely speaking, we have  \begin{align}
\IvKM(a+1,b;n|u):={}&\int_0^\infty I_{0}(\sqrt{u}t)[I_0(t)]^a[K_0(t)]^{b}t^{n}\D t,\\\IKvM(a,b+1;n|u):={}&\int_0^\infty K_{0}(\sqrt{u}t)[I_0(t)]^a[K_0(t)]^{b}t^{n}\D t,
\end{align} for certain non-negative integers $ a,b,n\in\mathbb Z_{\geq0}$ that make these integral expressions absolutely convergent for a given scaling parameter $u>0$. Differentiations in the variable $u$ will be denoted by short-hands like $ D^m
f(u):=\D^mf(u)/\D u^m$, where $m\in\mathbb Z_{\geq0}$. It is understood that $ D^0f(u)=f(u)$. For $N\in\mathbb Z_{>1}$, the Wro\'nskian determinant $ W[f_1(u),\dots,f_N(u)]$ refers to $ \det(D^{i-1}f_j(u))_{1\leq i,j\leq N}$.\eor\end{definition}
Here, for the convergence test of the two-scale Bessel moments, it would suffice to remind our readers of  the asymptotic expansions for the modified Bessel functions:\begin{align}
I_0(t)=\frac{e^{t}}{\sqrt{2\pi t}}\left[ 1+O\left( \frac{1}{t} \right) \right],\quad K_0(t)=\sqrt{\frac{\pi}{2t}}e^{-t}\left[ 1+O\left( \frac{1}{t} \right) \right],\label{eq:IK_asympt}
\end{align}as $ t\to\infty$. In the  $t\to0^+$ regime, the bounded term   $ I_{0}(t)=1+O(t^2)$ and  the mild singularity $K_0(t)=O(\log t)$ do not contribute to the convergence test of single-scale Bessel moments $ \IKM$ and their two-scale analogs $ \IvKM, \IKvM$. Later in this section, we will also find the following facts\begin{align}
\sup_{t>0} \sqrt tI_0(t)K_0(t)<\infty,\quad \sup_{t>0}t^{3}\left\vert [I_0(t)K_0(t)]^2-\frac{1}{4t^2} \right\vert<\infty\label{eq:IK_sup}\end{align}and\begin{align}\sup_{t>0}\frac{K_0(t)}{1+|\log t|}<\infty\label{eq:Klog_sup}
\end{align}useful in bound estimates for $ \IKvM(a,b+1;n|u)$, as $ u\to0^+$.

 Setting \begin{align}
\left\{\begin{array}{l}\mu^\ell_{2,1}(u)=\frac{\IvKM(1,4;2\ell-1|u)+4\IKvM(1,4;2\ell-1|u)}{5},\\\mu^\ell_{2,2}(u)=\IvKM(2,3;2\ell-1|u),\\\mu^\ell_{2,3}(u)=\IKvM(2,3;2\ell-1|u),\end{array}\right.
\end{align} we study the Wro\'nskian determinant\begin{align}\begin{split}
\varOmega_3(u):={}&W[\mu^1_{2,1}(u),\mu^1_{2,2}(u),\mu^1_{2,3}(u)]\\={}&\det \begin{pmatrix}D^{0}\mu^1_{2,1}(u)&D^{0}\mu^1_{2,2}(u)&D^{0}\mu^1_{2,3}(u) \\
D^1\mu^1_{2,1}(u)&D^1\mu^1_{2,2}(u)&D^1\mu^1_{2,3}(u) \\
D^2\mu^1_{2,1}(u)&D^2\mu^1_{2,2}(u)&D^2\mu^1_{2,3}(u) \\
\end{pmatrix}\end{split}\label{eq:3x3_Wdet}
\end{align} in the next lemma. \begin{lemma}[Vanhove differential equation for  $ \varOmega_3(u)$]\label{lm:W_ODE_M2}For $0<u<4$, the Wro\'nskian determinant $\varOmega_3(u):=W[\mu^1_{2,1}(u),\mu^1_{2,2}(u),\mu^1_{2,3}(u)]$ satisfies the following differential equation:\begin{align}
D^{1}\varOmega_3(u)=\frac{3\varOmega_3(u)}{2}D^1\log\frac{1}{u^2 (4-u) ( 16-u)}.\label{eq:W_ODE_M2_prep0}
\end{align}\end{lemma}\begin{proof}Using integration by parts in the variable $t$, one can verify that the following holonomic differential operator \cite[][Table 1, $n=4$]{Vanhove2014Survey}\begin{align}\begin{split}\widetilde L_3:={}&
u^2 (u - 4) (u - 16)D^3+6 u (u^2 - 15 u + 32)D^{2}\\&+(7 u^2 - 68 u + 64)D^{1}+(u-4)D^{0}\label{eq:VL3}
\end{split}\end{align} annihilates  every member of the set $\{\mu^1_{2,1}(u),\mu^1_{2,2}(u),\mu^1_{2,3}(u)\}$, for $u\in(0,4)$.

With the Kronecker delta\begin{align}
\delta_{i,j}=\begin{cases}0, & \text{if }i\neq j, \\
1, & \text{if }i=j, \\
\end{cases}
\end{align}we can show that \begin{align}\begin{split}&
D^1W[\mu^1_{2,1}(u),\mu^1_{2,2}(u),\mu^1_{2,3}(u)]=\sum_{k=1}^3\det(D^{i+\delta_{i,k}-1}\mu^1_{2,j}(u))_{1\leq i,j\leq 3}\\={}&\det(D^{i+\delta_{i,3}-1}\mu^1_{2,j}(u))_{1\leq i,j\leq 3}=\det \begin{pmatrix}D^{0}\mu^1_{2,1}(u)&D^{0}\mu^1_{2,2}(u)&D^{0}\mu^1_{2,3}(u) \\
D^1\mu^1_{2,1}(u)&D^1\mu^1_{2,2}(u)&D^1\mu^1_{2,3}(u) \\
D^3\mu^1_{2,1}(u)&D^3\mu^1_{2,2}(u)&D^3\mu^1_{2,3}(u) \\
\end{pmatrix}\\={}&-\frac{6 u (u^2 - 15 u + 32)}{u^2 (u - 4) (u - 16)}\det \begin{pmatrix}D^{0}\mu^1_{2,1}(u)&D^{0}\mu^1_{2,2}(u)&D^{0}\mu^1_{2,3}(u) \\
D^1\mu^1_{2,1}(u)&D^1\mu^1_{2,2}(u)&D^1\mu^1_{2,3}(u) \\
D^2\mu^1_{2,1}(u)&D^2\mu^1_{2,2}(u)&D^2\mu^1_{2,3}(u) \\
\end{pmatrix}.\end{split}\label{eq:W_ODE_M2_prep}
\end{align}Here, in the last step, we have subtracted linear combinations of the first two rows from the last row in the penultimate determinant, while appealing to the homogeneous differential equations $ \widetilde L_3\mu^1_{2,k}(u)=0$ for $ k\in\{1,2,3\},u\in(0,4)$. Clearly, the differential equation in   \eqref{eq:W_ODE_M2_prep} is equivalent to \eqref{eq:W_ODE_M2_prep0}.   \end{proof}\begin{remark}After carefully collecting boundary contributions to the Newton--Leibniz formula (see Lemma \ref{lm:VanhoveLn} for technical details), one can show that   $\widetilde L_{3}\IvKM(1,4;1|u)=-3$ holds for $u\in(0,16)$ and $ \widetilde L_{3}\IKvM(1,4;$ $1|u)=\frac{3}{4}$ holds for $u\in(0,  \infty)$. This justifies our choice of the particular linear combination in $\mu^1_{2,1}(u)=\frac{1}{5}\IvKM(1,4;1|u)+\frac{4}{5}\IKvM(1,4;1|u)$. The homogeneous differential equation $ \widetilde L_3\mu^1_{2,1}(u)=0$ was also crucially important in a previous study \cite[][\S5]{Zhou2017WEF} of the single-scale 6-loop sunrise diagram in two-di\-men\-sion\-al quantum field theory.\eor\end{remark}
\subsection{Reduction to $ \det\mathbf M_2$\label{subsec:detM2}}
We recall that the modified Bessel functions of first order are related to derivatives of their counterparts of zeroth order:\begin{align}
I_{1}(t)=\frac{\D I_0(t)}{\D t},\quad K_1(t)=-\frac{\D K_0(t)}{\D t},
\end{align}and we have a bound\begin{align}
\sup_{t>0}\frac{|tK_1(t)-1|}{t(1+|\log t|)}<\infty.\label{eq:K1_sup}
\end{align}Reserving the symbol $ D^1$ for partial derivatives in the variable $u$, we have \begin{align}
D^{1}I_0(\sqrt{u}t)=\frac{tI_{1}(\sqrt{u}t)}{2\sqrt{u}},\quad D^{1}K_0(\sqrt{u}t)=-\frac{tK_{1}(\sqrt{u}t)}{2\sqrt{u}}.
\end{align}This motivates us to introduce additional short-hand notations, to accommodate  for derivatives of two-scale Bessel moments $ \IvKM$ and $ \IKvM$ with respect to $u$.
\begin{definition}\label{defn:IpKM_IKpM}We write $ \IpKM$ (resp.~$ \IKpM$) for the replacement of    one $I_0(t)$ (resp.~$K_0(t)$) factor in the single-scale Bessel moments by one $ I_1(\sqrt{u}t)$ (resp.~$-K_1(\sqrt{u}t)$) factor. Concretely speaking, we define\begin{align}
\IpKM(a+1,b;n|u):={}&+\int_0^\infty I_1(\sqrt{u}t)[I_0(t)]^a[K_0(t)]^bt^{n+1}\D t,\\\IKpM(a,b+1;n|u):={}&-\int_0^\infty K_1(\sqrt{u}t)[I_0(t)]^a[K_0(t)]^bt^{n+1}\D t,
\end{align}for certain non-negative integers $ a,b,n\in\mathbb Z_{\geq0}$ that guarantee convergence of these  integrals  for a given  parameter $u>0$.     \eor\end{definition}
With the understanding that $ D^mf(1)=\D^mf(u)/\D u^m|_{u=1}$, we now investigate \begin{align}
\varOmega_3(1)={}&\det \begin{pmatrix}D^{0}\mu^1_{2,1}(1)&D^{0}\mu^1_{2,2}(1)&D^{0}\mu^1_{2,3}(1) \\
D^{1}\mu^1_{2,1}(1)&D^{1}\mu^1_{2,2}(1)&D^{1}\mu^1_{2,3}(1) \\
D^{2}\mu^1_{2,1}(1)&D^{2}\mu^1_{2,2}(1)&D^{2}\mu^1_{2,3}(1) \\
\end{pmatrix}.\end{align}To save space for matrix entries, we also define
\begin{align}
\left\{\begin{array}{l}\acute\mu^\ell_{2,1}(u)=\frac{\IpKM(1,4;2\ell-1|u)+4\IKpM(1,4;2\ell-1|u)}{5},\\\acute\mu^\ell_{2,2}(u)=\IpKM(2,3;2\ell-1|u),\\\acute\mu^\ell_{2,3}(u)=\IKpM(2,3;2\ell-1|u).\end{array}\right.
\end{align}\begin{proposition}[Factorization of   $\varOmega_3(1)$]\label{prop:W_M2_1}We have the following identity:\begin{align}
\varOmega_3(1)=\frac{\IKM(1,2;1)}{2^{3}}\det\mathbf M_2.\label{eq:W_detM2_fac}
\end{align}\end{proposition}\begin{proof}With the Bessel differential equations $ (uD^2+D^1)I_{0}(\sqrt{u}t)=\frac{t^2}{4}I_0(\sqrt{u}t)$ and $(uD^2+D^1)K_{0}(\sqrt{u}t)=\frac{t^2}{4}K_0(\sqrt{u}t) $, we can verify\begin{align}
2^{3}u^{3/2}\varOmega_3(u)=\det\begin{pmatrix}\mu^1_{2,1}(u) & \mu^1_{2,2}(u) & \mu^1_{2,3}(u) \\
\acute\mu^1_{2,1}(u) & \acute\mu^1_{2,2}(u) & \acute\mu^1_{2,3}(u) \\
\mu^2_{2,1}(u) & \mu^2_{2,2}(u) & \mu^2_{2,3}(u) \\
\end{pmatrix}\label{eq:Omega3u_alt}
\end{align}for all $ u\in(0,4)$, upon using elementary row operations.  In particular, we may identify $ 2^{3}\varOmega_3(1)$ with  \begin{align}
\det \begin{pmatrix}\IKM(1,4;1)&\IKM(2,3;1)&\IKM(2,3;1) \\
\acute\mu^1_{2,1}(1)&\acute\mu^1_{2,2}(1)&\acute\mu^1_{2,3}(1) \\
\IKM(1,4;3)&\IKM(2,3;3)&\IKM(2,3;3) \\
\end{pmatrix}.
\end{align}Now,  subtracting the second column from the last column in the  determinant above, while keeping in mind that $ I_0(t)K_1(t)+I_1(t)K_0(t)=\frac1t$ leads to $ \acute\mu^1_{2,3}(1)-\acute\mu^1_{2,2}(1)=-\IKM(1,2;1)$, we may  equate  $ 2^{3}\varOmega_3(1)$  with\begin{align}
\det \begin{pmatrix}\IKM(1,4;1)&\IKM(2,3;1)&0 \\
\acute\mu^1_{2,1}(1)&\acute\mu^1_{2,2}(1)&-\IKM(1,2;1) \\
\IKM(1,4;3)&\IKM(2,3;3)&0 \\
\end{pmatrix},\label{eq:detM2_fac_prep}
\end{align}thereby establishing our claim in  \eqref{eq:W_detM2_fac}.     \end{proof}

  In the next proposition, we examine the Wro\'nskian determinant in the $u\to0^+$ limit.
 \begin{proposition}[Factorization of $\varOmega_3(0^+)$]The limit\label{prop:W_M2_0}\begin{align}
\lim_{u\to0^+}
u^{3}\varOmega_3(u)=\frac{[\IKM(1,3;1)]^2}{2^{3}5}
\end{align}entails \begin{align}
\varOmega_3(u)=\frac{\pi^{4}}{2^{2}5[u^{2}(4-u) ( 16-u)]^{3/2}},\quad \forall u\in(0,4).\label{eq:W_detM2_eval}
\end{align}In particular, this implies the evaluation   $ \det \mathbf M_2=\frac{2\pi^{3}}{\sqrt{3^35^5}}$.\end{proposition}\begin{proof}From \eqref{eq:W_ODE_M2_prep0}, we know that $ [u^{2}(4-u) ( 16-u)]^{3/2}\varOmega_3(u)$ remains constant for $u\in(0,4)$. We will determine this constant by computing \begin{align}2^{9}\lim_{u\to0^+}
u^{3}\varOmega_3(u)\label{eq:W_detM2_lim}
\end{align}from\begin{align}
2^{3}u^{3}\varOmega_3(u)={}&\det \left(\begin{array}{rrr}\mu^1_{2,1}(u) & \mu^1_{2,2}(u) & \mu^1_{2,3}(u) \\
\sqrt{u}\acute\mu^1_{2,1}(u) &\sqrt{u} \acute\mu^1_{2,2}(u) &\sqrt{u} \acute\mu^1_{2,3}(u) \\
u\mu^2_{2,1}(u) & u\mu^2_{2,2}(u) & u\mu^2_{2,3}(u) \\
\end{array}\right).\label{eq:Omega3u_0_prefactor}
\end{align}

In the $u\to0^+$ regime, we have [cf.~\eqref{eq:IK_sup} and \eqref{eq:K1_sup}]\begin{align}\begin{split}
\mu^1_{2,3}(u)={}&\int_{0}^{\infty}K_0(\sqrt{u}t)[I_{0}(t)K_{0}(t)]^{2}t\D t\\={}&O\left( \int_{0}^{\infty}K_0(\sqrt{u}t)\D t \right)=O\left( \frac{1}{\sqrt{u}} \right),\label{eq:IKvM231}\end{split}\intertext{}\begin{split}-\sqrt{u}\IKpM(1,4;1|u)={}&\int_0^\infty I_0(t)[K_0(t)]^3 t\D t+\int_0^\infty [\sqrt{u}tK_{1}(\sqrt{u}t)-1]I_0(t)[K_0(t)]^3 t\D t\\={}&\IKM(1,3;1)+O(\sqrt{u}\log u),\end{split}\label{eq:K-1}
\end{align} along with several other asymptotic expansions, so $ 2^{3}u^{3}\varOmega_3(u)$ becomes\begin{align}
\det \begin{pmatrix}O(\log u)&\IKM(1,3;1)+O(u)&O(1/\sqrt{u}) \\
-\frac{4\IKM(1,3;1)}{5}+O(\sqrt{u}\log u)&O(u)&{\sqrt u}\acute\mu^1_{2,3}(u) \\
O(u\log u)&O(u)&u\mu^2_{2,3}(u) \\
\end{pmatrix}.\end{align}Noting that  [cf.~\eqref{eq:IK_sup}]\begin{align}\begin{split}
-{\sqrt u}\acute\mu^1_{2,3}(u)={}&\int_{0}^{\infty}\sqrt{u}K_1(\sqrt{u}t)[I_{0}(t)K_{0}(t)]^{2}t^{2}\D t\\={}&O\left( \int_{0}^{\infty}\sqrt{u}K_1(\sqrt{u}t)t\D t \right)=O\left( \frac{1}{\sqrt{u}} \right),\end{split}\label{eq:IKpM231}\intertext{and  [cf.~\eqref{eq:IK_sup}]}\begin{split}u\mu^2_{2,3}(u)={}&\frac{u}{4}\int_{0}^{\infty}K_0(\sqrt{u}t)t\D t+u\int_{0}^{\infty}K_0(\sqrt{u}t)\left\{ [I_{0}(t)K_{0}(t)]^{2}-\frac{1}{4t^{2}} \right\}t^{3}\D t\\={}&\frac{1}{4}\int_{0}^{\infty}K_0(t)t\D t+O\left( u \int_{0}^{\infty}K_0(\sqrt{u}t)\D t\right)\\={}&\frac{1}{4}+O(\sqrt{u}),\end{split}
\end{align} we find\begin{align}\begin{split}&
2^{3}u^{3}\varOmega_3(u)\\={}&
\det \begin{pmatrix}O(\log u)&\IKM(1,3;1)+O(u)&O(1/\sqrt{u}) \\
-\frac{4\IKM(1,3;1)}{5}+O(\sqrt{u}\log u)&O(u)&O(1/\sqrt{u}) \\
O(u\log u)&O(u)&\frac{1}{4}+O(\sqrt{u}) \\
\end{pmatrix}\\={}&\frac{[\IKM(1,3;1)]^2}{5}+O(\sqrt{u}\log u).\end{split}
\end{align} As we have $ \IKM(1,3;1)=\frac{\pi^2}{2^{4}}$  \cite[][(55)]{BBBG2008},  we see that the limit in \eqref{eq:W_detM2_lim} must be equal to $ \frac{\pi^4}{2^{2}5}$.

 Recalling the well-known evaluation $ \IKM(1,2;1)=\frac{\pi}{3\sqrt{3}}$ from \cite[][(23)]{BBBG2008}, we can compute $ \det \mathbf M_2=\frac{2\pi^{3}}{\sqrt{3^35^5}}$ with the aid of  \eqref{eq:W_detM2_fac} and \eqref{eq:W_detM2_eval}. \end{proof}
\section{An algebraic evaluation of $ \det \mathbf N_2$\label{sec:detN2}}

In \S\ref{sec:detM2}, we built $ \det \mathbf M_2$  on the knowledge of (the retroactively defined $ 1\times 1$ ``determinants'') $ \det\mathbf M_1=\IKM(1,2;1)$ and $ \det\mathbf N_1=\IKM(1,3;1)$. Our task in this section is to  compute $ \det \mathbf N_2$ from $\det\mathbf M_2$ and $\det \mathbf N_1$.\subsection{A $4\times 4$ Wro\'nskian determinant\label{subsec:W4x4}}Setting \begin{align}
\left\{\begin{array}{l}\nu^\ell_{2,1}(u)=\frac{\IvKM(1,5;2\ell-1|u)+5\IKvM(1,5;2\ell-1|u)}{6},\\\nu^\ell_{2,2}(u)=\IvKM(2,4;2\ell-1|u),\\\nu^\ell_{2,3}(u)=\IvKM(3,3;2\ell-1|u),\\\nu^\ell_{2,4}(u)=\IKvM(2,4;2\ell-1|u),\end{array}\right.
\end{align} and\begin{align}
\left\{\begin{array}{l}\acute\nu^\ell_{2,1}(u)=\frac{\IpKM(1,5;2\ell-1|u)+5\IKpM(1,5;2\ell-1|u)}{6},\\\acute\nu^\ell_{2,2}(u)=\IpKM(2,4;2\ell-1|u),\\\acute\nu^\ell_{2,3}(u)=\IpKM(3,3;2\ell-1|u),\\\acute\nu^\ell_{2,4}(u)=\IKpM(2,4;2\ell-1|u),\end{array}\right.
\end{align}we begin our study of  the Wro\'nskian determinant $ \omega_4(u):=W[\nu^1_{2,1}(u),\nu^1_{2,2}(u),\nu^1_{2,3}(u),\nu^1_{2,4}(u)]$ from the next lemma.\begin{lemma}[Vanhove differential equation for $ \omega_4(u)$]For $0<u<1$, the Wro\'nskian determinant $\omega_4(u):=W[\nu^1_{2,1}(u),\nu^1_{2,2}(u),\nu^1_{2,3}(u),\nu^1_{2,4}(u)]$ satisfies the following differential equation:\begin{align}
D^{1}\omega_4(u)=2\omega_4 (u)D^1\log\frac{1}{u^2 (1-u)(9-u) ( 25-u)}.\label{eq:W_ODE_M2_prep0}
\end{align}\end{lemma}\begin{proof}Using integration by parts in the variable $t$, one can verify that the following holonomic differential operator \cite[][Table 1, $n=5$]{Vanhove2014Survey}\begin{align}\begin{split}\widetilde L_4:={}&
u^2(u-25) (u-9) (u-1) D^{4}+2 u (5 u^3-140 u^2+777 u-450)D^3\\{}&+(25 u^3-518 u^2+1839 u-450)D^{2}+(3 u-5) (5 u-57)D^{1}+(u-5)D^{0}\label{eq:VL4}
\end{split}\end{align} annihilates every member in the set $ \{\nu^1_{2,1}(u),\nu^1_{2,2}(u),\nu^1_{2,3}(u),\nu^1_{2,4}(u)\}$. One may then proceed as in Lemma   \ref{lm:W_ODE_M2}. \end{proof}\begin{remark}We have   $\widetilde L_{4}\IvKM(1,5;1|u)=-\frac{15}{2}$ for $u\in(0,25)$ and $ \widetilde L_{4}\IKvM(1,5;1|u)=\frac{3}{2}$ for $u\in(0,  \infty)$. Such computations will be  put into a broader context in   Lemma \ref{lm:VanhoveLn}.\eor\end{remark}
\subsection{Reduction to $\det\mathbf N_2$\label{subsec:detN2}}
We now describe an analog of Proposition \ref{prop:W_M2_1}.
\begin{proposition}[Factorization of $\omega_4(1^-)$]\label{prop:W_N2_1}We have the following identity:\begin{align}
\lim_{u\to1^-}(1-u)^2\omega_4(u)=-\frac{\IKM(1,3;1)}{2^{7}}\det\mathbf N_2.\label{eq:W_detN2_fac}
\end{align}\end{proposition}\begin{proof}Through row operations and the Bessel differential equations for $I_0$ and $K_0$, we find\begin{align}
2^{6}u^{3}\omega_4(u)=\det\begin{pmatrix}\nu^1_{2,1}(u) & \nu^1_{2,2}(u) & \nu^1_{2,3}(u) &\nu^1_{2,4}(u)\\
\acute\nu^1_{2,1}(u) & \acute\nu^1_{2,2}(u) & \acute\nu^1_{2,3}(u)&\acute\nu^1_{2,4}(u) \\
\nu^2_{2,1}(u) & \nu^2_{2,2}(u) & \nu^2_{2,3}(u)&\nu^2_{2,4}(u) \\\acute\nu^2_{2,1}(u) & \acute\nu^2_{2,2}(u) & \acute\nu^2_{2,3}(u)&\acute\nu^2_{2,4}(u) \\
\end{pmatrix}
\end{align}for all $ u\in(0,1)$. In particular,  as $ u\to1^-$, we have \begin{align}\begin{split}
&2^{6}u^{3}\omega_4(u)\\={}&\det \begin{pmatrix}\IKM(1,5;1)+\circ&\nu^1_{2,2}(1)+\circ&\nu^1_{2,3}(1)+\circ&\nu^1_{2,4}(1)+\circ \\
\sharp\ &\acute \nu^1_{2,2}(1)+\circ &  \acute\nu^1_{2,3}(u) & \acute \nu^1_{2,4}(1)+\circ \\
\IKM(1,5;3)+\circ  & \nu^2_{2,2}(1)+\circ & \nu^2_{2,3}(1)&\nu^2_{2,4}(1)+\circ \\
\sharp\ & \sharp & \acute \nu^2_{2,3}(u) & \sharp \\
\end{pmatrix}\end{split}
\end{align}where a hash  (resp.~circle)  denotes a bounded (resp.~infinitesimal) quantity.
Here,
it is also worth pointing out that $ \nu^1_{2,2}(1)=\nu^1_{2,4}(1)=\IKM(2,4;1)$ and $ \nu^2_{2,2}(1)=\nu^2_{2,4}(1)=\IKM(2,4;3)$.

From a bound \begin{align}
\sup _{t>0}t^{2s}\left| [I_0(t)K_0(t)]^2-\frac{1}{4t^{2}} \right|<\infty,\quad s\in\mathbb \{1,2\}
\end{align}and generalized Weber--Schafheitlin integrals \cite[cf.][\S13.45]{Watson1944Bessel} for $ u\in(0,1)$:\begin{align}
\int_0^\infty I_0(\sqrt{u}t)K_0(t)t\D t={}&\frac{1}{1-u},\\\int_0^\infty I_1(\sqrt{u}t)K_0(t)\D t={}&-\frac{\log(1-u)}{2\sqrt{u}},\\\int_0^\infty I_1(\sqrt{u}t)K_0(t)t^{2}\D t={}&\frac{2\sqrt{u}}{(1-u)^{2}},
\end{align} we may deduce the following asymptotic formulae in the  $ u\to1^{-}$ regime:\begin{align}\begin{split}&\acute \nu^1_{2,3}(u)\\={}&\frac{1}{4}\int_0^\infty I_1(\sqrt{u}t)K_0(t)\D t+\int_0^\infty I_1(\sqrt{u}t)K_0(t)\left\{ [I_0(t)K_0(t)]^2-\frac{1}{4t^{2}}  \right\}t^2\D t\\={}&O(\log(1-u)),\end{split}\intertext{}\begin{split}&
(1-u)\nu^2_{2,3}(u)\\={}&\frac{(1-u)}{4}\int_0^\infty I_0(\sqrt{u}t)K_0(t)t\D t+(1-u)\int_0^\infty I_0(\sqrt{u}t)K_0(t)\left\{ [I_0(t)K_0(t)]^2-\frac{1}{4t^{2}}  \right\}t^3\D t\\={}&O(1),\end{split}\intertext{}\begin{split}&(1-u)^2\acute \nu^2_{2,3}(u)\\={}&\frac{(1-u)^{2}}{4}\int_0^\infty I_1(\sqrt{u}t)K_0(t)t^{2}\D t+(1-u)^{2}\int_0^\infty I_1(\sqrt{u}t)K_0(t)\left\{ [I_0(t)K_0(t)]^2-\frac{1}{4t^{2}}  \right\}t^4\D t\\={}&\frac{\sqrt{u}}{2}+O((1-u)^2\log(1-u)).\end{split}
\end{align}Therefore, we have \begin{align}\begin{split}
&2^{6}u^{2}(1-u)^{2}\omega_4(u)\\={}&\det \begin{pmatrix}\IKM(1,5;1)+\circ&\nu^1_{2,2}(1)+\circ&\circ&\nu^1_{2,4}(1)+\circ \\
\sharp\ &\acute \nu^1_{2,2}(1)+\circ & \circ & \acute \nu^1_{2,4}(1)+\circ \\
\IKM(1,5;3)+\circ  & \nu^2_{2,2}(1)+\circ & \circ&\nu^1_{2,4}(1)+\circ \\
\sharp\ & \sharp & \frac{1}{2}+\circ & \sharp \\
\end{pmatrix}\\={}&-\frac{1}{2}\det\begin{pmatrix}\IKM(1,5;1)+\circ&\IKM(2,4;1)+\circ&\IKM(2,4;1)+\circ \\
\sharp\ & \acute \nu^1_{2,2}(1)+\circ &  \acute \nu^1_{2,4}(1)+\circ \\
\IKM(1,5;3)+\circ  & \IKM(2,4;3)+\circ & \IKM(2,4;3)+\circ \\
\end{pmatrix}+o(1)\end{split}
\end{align}by cofactor expansion, as $ u\to1^-$. After eliminating the second column from the last column in the last $ 3\times3$ determinant, and employing $\acute \nu^1_{2,4}(1)-\acute \nu^1_{2,2}(1)=-\IKM(1,3;1)$, in a similar fashion as \eqref{eq:detM2_fac_prep}, we arrive at the factorization formula in \eqref{eq:W_detN2_fac}.
\end{proof}

Next, we consider an extension of Proposition \ref{prop:W_M2_0}.
\begin{proposition}[Factorization of  $\omega_4(0^+)$]\label{prop:W_N2_0}The limit\begin{align}
\lim_{u\to0^+}u^{4}\omega_4(u)=-\frac{5(\det \mathbf M_2)^{2}}{2^{7}3}
\end{align} entails \begin{align}
\omega_4(u)=-\frac{\pi^{6}}{2^{5}[u^2 (1-u)(9-u) ( 25-u)]^{2}},\quad \forall u\in(0,1).\label{eq:W_detN2_eval}
\end{align}In particular, this implies the evaluation   $ \det \mathbf N_2=\frac{\pi^{4}}{2^63^2}$.\end{proposition}\begin{proof}We will evaluate $ \lim_{u\to0^+}u^4\omega_4(u)$, starting from the expansion{\allowdisplaybreaks\begin{align}\begin{split}&
2^{6}u^{4}\omega_4(u)\\={}&\det\left(\begin{array}{rrrr}\nu^1_{2,1}(u) & \nu^1_{2,2}(u) & \nu^1_{2,3}(u) &\nu^1_{2,4}(u)\\
\sqrt{u}\acute\nu^1_{2,1}(u) & \sqrt{u}\acute\nu^1_{2,2}(u) &\sqrt{u} \acute\nu^1_{2,3}(u)&\sqrt{u}\acute\nu^1_{2,4}(u) \\
\nu^2_{2,1}(u) & \nu^2_{2,2}(u) & \nu^2_{2,3}(u)&\nu^2_{2,4}(u) \\\sqrt{u}\acute\nu^2_{2,1}(u) & \sqrt{u}\acute\nu^2_{2,2}(u) &\sqrt{u} \acute\nu^2_{2,3}(u)&\sqrt{u}\acute\nu^2_{2,4}(u) \\
\end{array}\right)\\={}&\det\begin{pmatrix}O(\log u) & \mu^1_{2,1}(1)+O(u) & \mu^1_{2,2}(1)+O(u) &O(\log u)\\
\sqrt{u}\acute\nu^1_{2,1}(u) & O(u) &O(u)&\sqrt{u}\acute\nu^1_{2,4}(u) \\
O(\log u) & \mu^2_{2,1}(1)+O(u) & \mu^2_{2,2}(1)+O(u)&O(\log u) \\\sqrt{u}\acute\nu^2_{2,1}(u) & O(u) &O(u)&\sqrt{u}\acute\nu^2_{2,4}(u) \\
\end{pmatrix}\\={}&-\det \begin{pmatrix}\IKM(1,4;1) & \IKM(2,3;1) \\
\IKM(1,4;3) & \IKM(2,3;3) \\
\end{pmatrix}\det\begin{pmatrix}\sqrt{u}\acute\nu^1_{2,1}(u) & \sqrt{u}\acute\nu^1_{2,4}(u) \\
\sqrt{u}\acute\nu^2_{2,1}(u) & \sqrt{u}\acute\nu^2_{2,4}(u) \\
\end{pmatrix}+O(u^{2}\log^2u),\end{split}
\end{align}}where $ \mu^\ell_{2,1}(1)=\IKM(1,4;2\ell-1)$ and $\mu^\ell_{2,2}(1)=\IKM(2,3;2\ell-1) $. Arguing in a similar vein as  \eqref{eq:K-1}, we find\begin{align}\begin{split}&
\begin{pmatrix}\sqrt{u}\acute\nu^1_{2,1}(u) & \sqrt{u}\acute\nu^1_{2,4}(u) \\
\sqrt{u}\acute\nu^2_{2,1}(u) & \sqrt{u}\acute\nu^2_{2,4}(u) \\
\end{pmatrix}\\={}&\begin{pmatrix}-\frac{5}{6}\IKM(1,4;1)+o(1) & -\IKM(2,3;1)+o(1) \\
-\frac{5}{6}\IKM(1,4;3)+o(1) & -\IKM(2,3;3)+o(1) \\
\end{pmatrix}\end{split}
\end{align}as $ u\to0^+$. Therefore, our goal is achieved.
 \end{proof}
\section{Broadhurst--Mellit formulae for $\det\mathbf M_k$ and $ \det\mathbf N_k$\label{sec:detMkNk}}The major goal of this section is to generalize the algebraic manipulations in \S\S\ref{sec:detM2}--\ref{sec:detN2} to the following recursions of  Broadhurst--Mellit determinants for all $ k\in\mathbb Z_{\geq2}$:
\begin{align} \det\mathbf M_{k-1}\det\mathbf M_k={}&\frac{k [\Gamma (k/2)]^2(\det\mathbf N_{k-1})^2}{2 (2 k+1)}\prod _{j=1}^k \left[\frac{(2 j)^2}{(2 j)^2-1}\right]^{k-\frac{1}{2}},\label{eq:detM_rec}\\\det\mathbf N_{k-1}\det\mathbf N_k={}&\frac{2k+1}{k+1}\frac{(\det\mathbf M_{k})^2}{(k-1)!}\prod _{j=2}^{k+1} \left[\frac{(2 j-1)^2}{(2 j-1)^2-1}\right]^{k}.\label{eq:detN_rec}\end{align}  Once these recursions are established, we can verify Conjectures \ref{conj:BMdetM} and \ref{conj:BMdetN} by induction. \subsection{Wro\'nskians for two-scale Bessel moments}The analysis in \S\S\ref{sec:detM2}--\ref{sec:detN2} motivates us to introduce the following notations for matrix elements.
\begin{definition}For each $ k\in\mathbb Z_{\geq2}$, we set \begin{align}
\left\{\begin{array}{l}\mu^\ell_{k,1}(u)=\frac{\IvKM(1,2k;2\ell-1|u)+2k\IKvM(1,2k;2\ell-1|u)}{2k+1},\\\mu^\ell_{k,j}(u)=\IvKM(j,2k+1-j;2\ell-1|u),\forall j\in\mathbb Z\cap[2,k],\\\mu^\ell_{k,j}(u)=\IKvM(j-k+1,3k-j;2\ell-1|u),\forall j\in\mathbb Z\cap[k+1,2k-1],\end{array}\right.
\end{align}and \begin{align}
\left\{\begin{array}{l}\nu^\ell_{k,1}(u)=\frac{\IvKM(1,2k+1;2\ell-1|u)+(2k+1)\IKvM(1,2k+1;2\ell-1|u)}{2(k+1)},\\\nu^\ell_{k,j}(u)=\IvKM(j,2k+2-j;2\ell-1|u),\forall j\in\mathbb Z\cap[2,k+1],\\\nu^\ell_{k,j}(u)=\IKvM(j-k,3k+2-j;2\ell-1|u),\forall j\in\mathbb Z\cap[k+2,2k].\end{array}\right.
\end{align}For $a,b\in\mathbb Z\cap[1,k]$, we also write $ \mu_{k,a}^b=\mu_{k,a}^b(1)$ and $ \nu_{k,a}^{b}=\nu_{k,a}^b(1)$, as the abbreviations for the entries in the Broadhurst--Mellit matrices:\begin{align}
\mu_{k,a}^b={}&(\mathbf M_k)_{a,b}:=\int_0^\infty[I_0(t)]^a[K_0(t)]^{2k+1-a}t^{2b-1}\D t,\\\nu_{k,a}^b={}&(\mathbf N_k)_{a,b}:=\int_0^\infty[I_0(t)]^a[K_0(t)]^{2k+2-a}t^{2b-1}\D t.
\end{align}  For each $ k\in\mathbb Z_{\geq2}$, we will be concerned with  \begin{align} \varOmega_{2k-1}(u):={}&W[\mu^1_{k,1}(u),\dots,\mu^1_{k,2k-1}(u)],\\\omega_{2k}(u):={}&W[\nu^1_{k,1}(u),\dots,\nu^1_{k,2k}(u)],\end{align}  the Wro\'nskian determinants for two-scale Bessel moments. \eor\end{definition}

 If we further define\begin{align}
\left\{\begin{array}{l}\acute\mu^\ell_{k,1}(u)=\frac{\IpKM(1,2k;2\ell-1|u)+2k\IKpM(1,2k;2\ell-1|u)}{2k+1},\\\acute\mu^\ell_{k,j}(u)=\IpKM(j,2k+1-j;2\ell-1|u),\forall j\in\mathbb Z\cap[2,k],\\\acute\mu^\ell_{k,j}(u)=\IKpM(j-k+1,3k-j;2\ell-1|u),\forall j\in\mathbb Z\cap[k+1,2k-1],\end{array}\right.
\end{align}and  \begin{align}
\left\{\begin{array}{l}\acute\nu^\ell_{k,1}(u)=\frac{\IpKM(1,2k+1;2\ell-1|u)+(2k+1)\IKpM(1,2k+1;2\ell-1|u)}{2(k+1)},\\\acute\nu^\ell_{k,j}(u)=\IpKM(j,2k+2-j;2\ell-1|u),\forall j\in\mathbb Z\cap[2,k+1],\\\acute\nu^\ell_{k,j}(u)=\IKpM(j-k,3k+2-j;2\ell-1|u),\forall j\in\mathbb Z\cap[k+2,2k],\end{array}\right.
\end{align}then we can verify\begin{align}
(2\sqrt{u})^{(k-1)(2k-1)}\varOmega_{2k-1}(u)={}&\det \begin{pmatrix}\mu_{k,1}^1(u) & \cdots\ & \mu_{k,2k-1}^1(u) \\
\acute\mu_{k,1}^1(u) & \cdots\ & \acute \mu_{k,2k-1}^1(u) \\
\multicolumn{3}{c}{\cdots\cdots\cdots\cdots\cdots\cdots\cdots\cdots}\ \\
\mu_{k,1}^k(u) & \cdots\ & \mu_{k,2k-1}^k(u)  \\
\end{pmatrix}\label{eq:varOmega_2k_alg}\intertext{for $u\in(0,4) $, and}
(2\sqrt{u})^{(2k-1)k}\omega_{2k}(u)={}&\det\begin{pmatrix}\nu_{k,1}^1(u) & \cdots\ & \nu_{k,2k-1}^1(u) \\
\acute\nu_{k,1}^1(u) & \cdots\ & \acute \nu_{k,2k-1}^1(u) \\
\multicolumn{3}{c}{\cdots\cdots\cdots\cdots\cdots\cdots\cdots\cdots} \\
\nu_{k,1}^k(u) & \cdots\ & \nu_{k,2k-1}^k(u)  \\
\acute
\nu_{k,1}^k(u) & \cdots\ &\acute \nu_{k,2k-1}^k(u)\\
\end{pmatrix}\label{eq:omega_2k_alg}\end{align}
for $u\in(0,1) $, through iterated applications of the Bessel differential equations $ (uD^2+D^1)I_{0}(\sqrt{u}t)=\frac{t^2}{4}I_0(\sqrt{u}t)$ and $(uD^2+D^1)K_{0}(\sqrt{u}t)=\frac{t^2}{4}K_0(\sqrt{u}t) $.\begin{lemma}[Vanhove differential equations for  $ \varOmega_{2k-1}(u)$ and $ \omega_{2k}(u)$]\label{lm:VanhoveLn}\begin{enumerate}[leftmargin=*,  label=\emph{(\alph*)},ref=(\alph*),
widest=a, align=left]\item For each $ n\in\mathbb Z_{\geq1}$, there exists a holonomic differential operator $ \widetilde L_n$ whose leading term is $ f_n(u)D^n$, such that $ f_n(u)$ is a monic polynomial and \begin{align}
\begin{cases}\widetilde L_n\IvKM(1,n+1,1|u)=-\frac{(n+1)!}{2^{n}}, &  \\
\widetilde L_n\IKvM(1,n+1,1|u)=\frac{n!}{2^{n}}, &  \\\widetilde
L_n\IvKM(j,n+2-j,1|u)=0, & \forall j\in\mathbb Z\cap[2,\frac{n}{2}+1], \\ \widetilde L_n\IKvM(j,n+2-j,1|u)=0,&\forall j\in\mathbb Z\cap[2,\frac{n+1}{2}].
\end{cases}\label{eq:VanhoveLnODE}
\end{align}  \item For $u\in(0,4)$, we have\begin{align}
D^{1}\varOmega_{2k-1}(u)={}&\frac{2k-1}{2}\varOmega_{2k-1}(u)D^1\log\frac{1}{u^k \prod _{j=1}^k[(2j)^2-u]};\label{eq:W_ODE_Mk_prep0}\intertext{for $u\in(0,1) $, we have}D^{1}\omega_{2k}(u)={}&k\omega_{2k}(u)D^1\log\frac{1}{u^k \prod _{j=1}^{k+1}[(2j-1)^2-u]}.\label{eq:W_ODE_Nk_prep0}
\end{align} \end{enumerate}\end{lemma}\begin{proof}\begin{enumerate}[leftmargin=*,  label=(\alph*),ref=(\alph*),
widest=a, align=left]\item With the notations $ \eth^0f(t)=f(t)$ and $ \eth^{n+1}f(t)=t\frac{\D}{\D t}\eth^nf(t)$ for all $ n\in\mathbb Z_{\geq0}$, we have the Bessel differential equations $ \eth^2I_0(t)=t^2\eth^0I_0(t)$  and   $ \eth^2K_0(t)=t^2\eth^0K_0(t)$. The Borwein--Salvy operator $ L_{n+1}$ \cite[][Lemma 3.3]{BorweinSalvy2007}, being the $n$-th symmetric power of the Bessel differential operator $ \eth^2-t^2\eth^0$, annihilates each member in the set $ \{[I_0(t)]^j[K_0(t)]^{n-j}|j\in\mathbb Z\cap[0,n]\}$. The Borwein--Salvy operator $L_{n+1}=\mathscr L_{n+1,n+1}$ can be constructed by the Bronstein--Mulders--Weil algorithm \cite[][Theorem 1]{BMW1997}:\begin{align}
\begin{cases}\mathscr L_{n+1,0}=\eth^0,\mathscr L_{n+1,1}=\eth^1, &  \\
\mathscr L_{n+1,k+1}=\eth^1\mathscr L_{n+1,k}-k(n+1-k)t^{2}\mathscr L_{n+1,k-1}, &          \forall k\in\mathbb Z\cap[1,n]. \\
\end{cases}\label{eq:BMW}
\end{align}For each fixed $j\in\mathbb Z\cap[0,n]$, one can use the aforementioned recursion for the operators $ \mathscr L_{n+1,k}$, the Leibniz rule for derivatives, and the Bessel differential equation, to prove a  formula \cite[cf.][Theorem 1]{BMW1997}
\begin{align}\begin{split}&
\mathscr L_{n+1,k}\{[I_0(t)]^j[K_0(t)]^{n-j}\}\\={}&\sum_{\ell=0}^{k}\frac{k!}{\ell!(k-\ell)!}\frac{j!}{(j-\ell)!}\frac{(n-j)!}{(n-j-k+\ell)!}[\eth^{1} I_0(t)]^\ell[I_0(t)]^{j-\ell}[\eth ^{1}K_0(t)]^{k-\ell}[K_0(t)]^{n-j-k+\ell}\label{eq:L_descent}
\end{split}\end{align}by induction on $ k\in\mathbb Z\cap[0,n]$.
(Here, we need the convention  $ 1/(-m)!=0$ for all positive integers $ m$.)
In particular, we have the following identities for $ k\in\mathbb Z\cap[0,n]$ \cite[cf.][Lemma 3.1]{BorweinSalvy2007} \begin{align}
 \mathscr L_{n+1,k}\{[K_0(t)]^n\}={}&\frac{n!}{(n-k)!}[K_0(t)]^{n-k}[\eth^1 K_0(t)]^{k},\label{eq:Kn_descent}\\\begin{split}\mathscr L_{n+1,k}\{I_{0}(t)[K_0(t)]^{n-1}\}={}&\frac{ (n-1)!k}{(n-k)!}[\eth^1 I_0(t)][K_0(t)]^{n-k}[\eth^1 K_0(t)]^{k-1}\\{}&+\frac{ (n-1)!}{(n-k-1)!}I_0(t)[K_0(t)]^{n-k-1}[\eth^1 K_0(t)]^{k}.\label{eq:IKn_1_descent}
\end{split}\end{align}

Once we have obtained \begin{align}
L_{n+1}=\sum_{k=0}^{n+1}\lambda_{n+1,k}(t)\frac{\partial^k}{\partial t^k}
\end{align}from the  Bronstein--Mulders--Weil algorithm described  above [with the understanding that $ \frac{\partial^0}{\partial t^0}g(t,u)=g(t,u)$], we can define the action of its formal adjoint $ L^*_{n+1}$ on a bivariate function $g(t,u)$ as follows:\begin{align}
L^*_{n+1}g(t,u)=\sum_{k=0}^{n+1}(-1)^{k}\frac{\partial^k}{\partial t^k}[\lambda_{n+1,k}(t)g(t,u)].
\end{align}
The design of  Vanhove's operators $ \widetilde L_n,n\in\mathbb Z_{\geq1}$ in \cite[][\S9]{Vanhove2014Survey} ensures that \begin{align}
\left\{ \begin{array}{c}
t\widetilde L_nI_0(\sqrt{u}t)=\frac{(-1)^n}{2^n}L^*_{n+2}\frac{I_0(\sqrt{u}t)}{t}, \\
t\widetilde L_nK_0(\sqrt{u}t)=\frac{(-1)^n}{2^n}L^*_{n+2}\frac{K_0(\sqrt{u}t)}{t}. \\
\end{array} \right.
\end{align}

Starting from the vanishing identity \begin{align}
0=\int_0^\infty\frac{I_0(\sqrt{u}t)}{t}L_{n+1}\{[K_0(t)]^n\}\D t,\label{eq:sunrise_L}
\end{align}we may perform successive integrations by parts, while carefully treating boundary contributions from the $t\to0^+$ regime. We recall the recursion $ L_{n+1}=\mathscr L_{n+1,n+1}=\eth^1\mathscr L_{n+1,n}-nt^{2}\mathscr L_{n+1,n-1}$ from  \eqref{eq:BMW} and the closed-form formula for $ \mathscr L_{n+1,k}\{[K_0(t)]^n\}$ from \eqref{eq:Kn_descent}. These identities enable us to rewrite  \eqref{eq:sunrise_L} as \begin{align}\begin{split}
0={}&\int_0^\infty I_0(\sqrt{u}t)\frac{\partial}{\partial t}\mathscr L_{n+1,n}\{[K_0(t)]^n\}\D t-n\int_0^\infty tI_0(\sqrt{u}t)\mathscr L_{n+1,n-1}\{[K_0(t)]^n\}\D t\\={}&-(-1)^{n}n!-\int_0^\infty\mathscr  L_{n+1,n}\{[K_0(t)]^n\}\frac{\partial I_0(\sqrt{u}t)}{\partial t}\D t\\{}&-n\int_0^\infty tI_0(\sqrt{u}t)\mathscr L_{n+1,n-1}\{[K_0(t)]^n\}\D  t,\end{split}\label{eq:int_part_sunrise}
\end{align}where the boundary contribution comes from $ \lim_{t\to0^+}\mathscr L_{n+1,n}\{[K_0(t)]^n\}=n!\lim_{t\to0^+}[-tK_1(t)]^n=(-1)^n $ $n!$. None of the  subsequent  integrations by parts will   incur any non-vanishing boundary contributions, because we have $ \lim_{t\to0^+}t^\ell\log^mt=0$ for all $ \ell,m\in\mathbb Z_{>0}$. Thus, we can recast \eqref{eq:int_part_sunrise} into \begin{align}\begin{split}
0={}&-(-1)^{n}n!+\int_0^\infty [K_0(t)]^nL_{n+1}^*\frac{I_0(\sqrt{u}t)}{t}\D t\\={}&-(-1)^{n}n!+(-1)^{n-1}2^{n-1}\widetilde L_{n-1}\IvKM(1,n,1|u),\end{split}
\end{align}which proves the first identity in \eqref{eq:VanhoveLnODE}.

In a similar vein, we may integrate by parts with the help from   \eqref{eq:BMW} and \eqref{eq:IKn_1_descent}:\begin{align}\begin{split}
0={}&\int_0^\infty\frac{K_0(\sqrt{u}t)}{t}L_{n+1}\{I_0(t)[K_0(t)]^{n-1}\}\D t\\={}&-\int_0^\infty\mathscr  L_{n+1,n}\{I_{0}(t)[K_0(t)]^{n-1}\}\frac{\partial K_0(\sqrt{u}t)}{\partial t}\D t-n\int_0^\infty tK_0(\sqrt{u}t)\mathscr L_{n+1,n-1}\{I_{0}(t)[K_0(t)]^{n-1}\}\D  t\\={}&\lim_{t\to0^+}\left(t\frac{\partial K_0(\sqrt{u}t)}{\partial t}\mathscr  L_{n+1,n-1}\{I_{0}(t)[K_0(t)]^{n-1}\}\right)+(-1)^{n-1}2^{n-1}\widetilde L_{n-1}\IKvM(1,n,1|u)\\={}&(-1)^n(n-1)!+(-1)^{n-1}2^{n-1}\widetilde L_{n-1}\IKvM(1,n,1|u),\end{split}
\end{align} which proves the second identity in  \eqref{eq:VanhoveLnODE}.

All the remaining cases in  \eqref{eq:VanhoveLnODE} can be proved by examining the asymptotic behavior of \eqref{eq:L_descent} in the $t\to0^+$ regime.

\item

From \eqref{eq:VanhoveLnODE}, we know that  for each $ k\in\mathbb Z_{\geq2}$, Vanhove's operator   $ \widetilde L_{2k-1}$ (resp.~$ \widetilde L_{2k}$) annihilates every member in the set $ \{\mu^1_{k,j}(u)|j\in\mathbb Z\cap[1,2k-1]\}$ (resp.~$\{\nu^1_{k,j}(u)|j\in\mathbb Z\cap[1,2k]\} $).

 For $ k\in\mathbb Z_{\geq2}$, Vanhove's operators $ \widetilde L_{2k-1}$ and $\widetilde L_{2k} $ take the following forms  \cite[(9.11)--(9.12)]{Vanhove2014Survey}:\begin{align}
\widetilde L_{2k-1}={}&\mathfrak{m}_{2k-1}(u)D^{2k-1}+\frac{2k-1}{2}\frac{\D\mathfrak{m}_{2k-1}(u)}{\D u}D^{2k-2}+L.O.T.,\\\widetilde L_{2k}={}&\mathfrak n_{2k}(u)D^{2k}+k\frac{\D\mathfrak{n}_{2k}(u)}{\D u}D^{2k-1}+L.O.T.,
\end{align}where\begin{align}
\mathfrak{m}_{2k-1}(u)=u^k \prod _{j=1}^k[u-(2j)^2],\quad \mathfrak n_{2k}(u)=u^k \prod _{j=1}^{k+1}[u-(2j-1)^2],\label{eq:m_n_poly}
\end{align}and ``$ L.O.T.$'' stands for ``lower order terms''.
Therefore, the corresponding Wro\'nskians must evolve according to \eqref{eq:W_ODE_Mk_prep0} and \eqref{eq:W_ODE_Nk_prep0}.
\qedhere\end{enumerate} \end{proof}\begin{remark}Prior to the work of Vanhove \cite{Vanhove2014Survey}, various authors  \cite{LR2004,MSWZ2012,ABW2014} have considered the operator $ \widetilde L_2$. Although Vanhove formulated his theory in \cite[][\S9]{Vanhove2014Survey} only for ``sunrise diagrams'' $\IvKM(1,n;1|u) $, his ideas generalize well to Feynman graphs with other topologies, as indicated in the proof above. For an extension of Vanhove's differential equations to quantum field theory in arbitrary dimensions, see M\"uller-Stach--Weinzierl--Zayadeh \cite{MSWZ2014}.  \eor\end{remark}\begin{remark}For $n\in\mathbb Z_{>1} $, Kluyver's  function $ p_n(x)=\int_0^\infty J_0(xt)[J_0(t)]^n xt\D t$ represents the probability density  for the distance traveled by a random walker in the Euclidean plane after  $n$ consecutive unit steps aiming at random directions. Here,   $ J_0(t):=\frac{2}{\pi}\int_0^{\pi/2}\cos(t\cos\varphi)\D\varphi$ is the Bessel function of the first kind.
It has been shown by Borwein--Straub--Wan--Zudilin  that $p_n(x)$ is holonomic, whose annihilator has the  form $ g_{n}(x)\frac{\D^{n-1}}{\D x^{n-1}}+L.O.T.$ where \cite[][(2.8)]{BSWZ2012}\begin{align}
g_{n}(x)=x^{n-1}\prod_{\substack {m\in\mathbb Z\cap[1,n]\\m\equiv n\hspace{-0.5em}\pmod2}}(x^2-m^2).\label{eq:g_poly}
\end{align}   The resemblance between \eqref{eq:m_n_poly} and \eqref{eq:g_poly} is not accidental. We refer our readers to \cite{Zhou2017PlanarWalks} for the connection between Kluyver's  probability density function and two-scale Bessel moments.\eor\end{remark}\subsection{Reduction of $ \det \mathbf M_k$ to $ \det\mathbf M_{k-1}$ and $ \det\mathbf N_{k-1}$}Now we factorize
  $\varOmega_{2k-1} $ in a similar spirit as Propositions \ref{prop:W_M2_1} and \ref{prop:W_M2_0}.
\begin{proposition}[Factorization of $ \varOmega_{2k-1}(1)$]For each $ k\in\mathbb Z_{\geq2}$, we have\begin{align}
\varOmega_{2k-1}(1)=(-1)^{\frac{(k-1)(k-2)}{2}}\frac{\det\mathbf M_{k-1}}{2^{(k-1)(2k-1)}}\det\mathbf M_k.\label{eq:detMk_fac}
\end{align} \end{proposition}
\begin{proof}In the formula\begin{align}\begin{split}&
2^{(k-1)(2k-1)}\varOmega_{2k-1}(1)\\={}&\det \begin{pmatrix}\mu_{k,1}^1(1) & \cdots&\mu^1_{k,k}(1)&\mu^1_{k,k+1}(1)&\cdots & \mu_{k,2k-1}^1(1) \\
\acute\mu_{k,1}^1(1) & \cdots&\acute\mu^1_{k,k}(1)&\acute\mu^1_{k,k+1}(1)&\cdots  & \acute \mu_{k,2k-1}^1(1) \\
\multicolumn{6}{c}{\cdots\cdots\cdots\cdots\cdots\cdots\cdots\cdots\cdots\cdots\cdots\cdots\cdots\cdots\cdots\cdots}   \\
\mu_{k,1}^k(1) & \cdots\ & \mu^k_{k,k}(1)&\mu^k_{k,k+1}(1)&\cdots &\mu_{k,2k-1}^k(1)  \\
\end{pmatrix},\end{split}
\end{align}we observe that \begin{align}
\left\{ \begin{array}{l}
\mu^\ell_{k,j}(1)=\mu^\ell_{k,k+j-1}(1)=\mu^\ell_{k,j}, \\\acute
\mu^\ell_{k,k+j-1}(1)-\acute \mu^\ell_{k,j}(1)=-\mu_{k-1,j-1}^\ell
\end{array} \right.
\end{align}for all $  j\in\mathbb Z\cap[2,k]$. Thus, we obtain, after column eliminations and row bubble sorts,\begin{align}\begin{split}&
2^{(k-1)(2k-1)}\varOmega_{2k-1}(1)\\={}&\det \begin{pmatrix}\mu_{k,1}^1 & \cdots&\mu^1_{k,k}&0&\cdots & 0 \\
\acute\mu_{k,1}^1(1) & \cdots&\acute\mu^1_{k,k}(1)&-\mu_{k-1,1}^1&\cdots  &  -\mu_{k-1,k-1}^1 \\
\multicolumn{6}{c}{\cdots\cdots\cdots\cdots\cdots\cdots\cdots\cdots\cdots\cdots\cdots\cdots\cdots\cdots\cdots\cdots} \\
\mu_{k,1}^k & \cdots\ & \mu^k_{k,k}&0&\cdots &0  \\
\end{pmatrix}\\={}&(-1)^{\frac{k(k-1)}{2}}\det\left( \begin{array}{c|c}\begin{array}{ccc}
 &  &  \\
 & \raisebox{-0.25\height}{\resizebox{1.75\width}{1.75\height}{$\mathbf M_k^{\mathrm T}$}} &  \\
 &  &  \\
\end{array}&\raisebox{-0.25\height}{\resizebox{1.75\width}{1.75\height}{$\mathbf O_{}$}}\\\hline\begin{array}{ccc}\vspace{-0.75em}\\\acute
\mu_{k,1}^1(1) & \cdots&\acute\mu^1_{k,k}(1) \\
\multicolumn{3}{c}{\cdots\cdots\cdots\cdots\cdots\cdots\cdots\cdots}  \\
\acute
\mu_{k,1}^{k-1}(1) & \cdots&\acute\mu^{k-1}_{k,k}(1) \\
\end{array}&\raisebox{-0.25\height}{\resizebox{1.75\width}{1.75\height}{$-\mathbf M_{k-1}^{\mathrm T}$}}\end{array} \right),\end{split}
\end{align}which factorizes as claimed.\end{proof}

\begin{proposition}[Factorization of $ \varOmega_{2k-1}(0^{+})$]The limit\begin{align}
\lim_{u\to0^+}u^{k(2k-1)/2}\varOmega_{2k-1}(u)=(-1)^{\frac{(k-1)(k-2)}{2}}\frac{k[\Gamma(k/2)]^{2}}{(2k+1)}\frac{(\det\mathbf N_{k-1})^2}{2^{(k-1)(2k-1)+1}}
\end{align}entails\begin{align}
\varOmega_{2k-1}(u)=\frac{(-1)^{\frac{(k-1)(k-2)}{2}}k[\Gamma(k/2)]^{2}}{u^{k(2k-1)/2}(2k+1)}\frac{(\det\mathbf N_{k-1})^2}{2^{(k-1)(2k-1)+1}}\prod _{j=1}^k\left[\frac{(2j)^2}{(2j)^2-u}\right]^{k-\frac{1}{2}},\quad \forall u\in(0,4).\label{eq:Omega_2k_1_u_rational}
\end{align}\end{proposition}
\begin{proof}As we compare the representation \begin{align}\begin{split}&
2^{(k-1)(2k-1)}u^{k(2k-1)/2}\varOmega_{2k-1}(u)\\={}&\det\left( \begin{array}{c|c}\begin{array}{rrr}\phantom{u^{k/2}}\mu_{k,1}^1(u) & \cdots\ & \phantom{u^{k/2}}\mu_{k,2k-2}^1(u) \\
\sqrt{u}\acute\mu_{k,1}^1(u) & \cdots\ & \sqrt{u}\acute \mu_{k,2k-2}^1(u) \\
\multicolumn{3}{c}{\cdots\cdots\cdots\cdots\cdots\cdots\cdots\cdots\cdots}  \ \\
\end{array}&\begin{array}{r}\phantom{u^{k/2}}\mu_{k,2k-1}^1(u)\\\sqrt{u}\acute \mu_{k,2k-1}^1(u)\\\cdots\cdots\cdots\cdots\cdots\end{array}\\\hline\begin{array}{rrr}\phantom{\sqrt{u}\acute\mu_{k,1}^1(u) }& \phantom{\cdots} & \phantom{\sqrt{u}\acute \mu_{k,2k-2}^1(u)}\\[-10pt]u^{k/2}\mu_{k,1}^k(u) & \cdots\ & u^{k/2}\mu_{k,2k-2}^k(u)\end{array}&\begin{array}{r}\phantom{\sqrt{u}\acute \mu_{k,2k-1}^1(u)}\\[-10pt]u^{k/2}\mu_{k,2k-1}^k(u)\end{array}\end{array}\right),\end{split}\label{eq:Omega_2k_1_0_fac}
\end{align}with \eqref{eq:varOmega_2k_alg}, we see that each row involving $ \acute \mu_{k,j}^\ell$ now bears an additional pre-factor of $ \sqrt u$; the first $(k-1)$ rows involving $\mu_{k,j}^\ell$ are left intact, but the bottom row in \eqref{eq:varOmega_2k_alg} is multiplied by a factor of $u^{k/2}$. Clearly, this setting hearkens back to \eqref{eq:Omega3u_0_prefactor}.

Akin to what we had in Proposition \ref{prop:W_M2_0} when $u$ is a positive infinitesimal, we can establish  the following asymptotic behavior of the first $(2k-2)$  columns in \eqref{eq:Omega_2k_1_0_fac}:\begin{align}\mu_{k,j}^\ell(u)=
\begin{cases}O(\log u), & j\in\{1\}\cup(\mathbb Z\cap[k+1,2k-2]) \\
\nu^\ell_{k-1,j-1}+O(u), & j\in\mathbb Z\cap[2,k]
\end{cases}\label{eq:mu_0_est}\end{align}for $ \ell\in\mathbb Z\cap[1,k]$, and \begin{align}
\sqrt{u}\acute \mu_{k,j}^\ell(u)=\begin{cases}-\frac{2k}{2k+1}\nu^\ell_{k-1,1}+o(1), & j=1 \\
O(u), & j\in\mathbb Z\cap[2,k] \\
-\nu^{\ell}_{k-1,j-k-1}+o(1), & j\in\mathbb Z\cap[k+1,2k-2] \\
\end{cases}
\end{align} for $ \ell\in\mathbb Z\cap[1,k-1]$. Here, it is understood that when $ k=2$, the closed interval $ [k+1,2k-2]=[3,2]=\varnothing$ is the empty set, so $ \{1\}\cup(\mathbb Z\cap[k+1,2k-2])$ degenerates to $ \{1\}$ in this scenario. We also bear in mind that the bottom row in \eqref{eq:Omega_2k_1_0_fac} carries an additional factor of $ u^{k/2}$, so  the estimate in \eqref{eq:mu_0_est} tells us that  the bottom-left section of the partitioned matrix in  \eqref{eq:Omega_2k_1_0_fac} contains only infinitesimal elements, with order at most $O(u^{k/2}\log u)$.

Meanwhile, we point out that the top-right block in    \eqref{eq:Omega_2k_1_0_fac} contains elements of order  $ O(1/\sqrt{u}) $, according to the rationale in \eqref{eq:IKvM231} and \eqref{eq:IKpM231}. The bottom-right element behaves like\begin{align}\begin{split}
u^{k/2}\mu_{k,2k-1}^k(u)={}&\frac{u^{k/2}}{2^{k}}\int_0^\infty K_0(\sqrt{u}t)t^{k-1}\D t\\{}&+u^{k/2}\int_{0}^\infty K_0(\sqrt{u}t)\left\{[I_{0}(t)K_{0}(t)]^{k}-\frac{1}{(2t)^{k}}\right\}t^{2k-1}\D t\\={}&\frac{1}{2^{k}}\int_0^\infty K_0(t)t^{k-1}\D t+O\left( u^{k/2} \int_{0}^\infty K_0(\sqrt{u}t)t^{k-2}\D t\right)\\={}&\frac{[\Gamma(k/2)]^{2}}{4}+O(\sqrt{u}),\end{split}
\end{align}where we have quoted the evaluation of  $ \int_0^\infty K_0(t)t^{k-1}\D t$ from Heaviside's integral formula \cite[][\S13.21(8)]{Watson1944Bessel}.

After taking care of the sign changes due to row and column permutations, we conclude that \begin{align}&
2^{(k-1)(2k-1)}\lim_{u\to0^+}u^{k(2k-1)/2}\varOmega_{2k-1}(u)=(-1)^{\frac{(k-1)(k-2)}{2}}\frac{k[\Gamma(k/2)]^{2}}{2(2k+1)}(\det\mathbf N_{k-1})^2
\end{align}   as claimed. \end{proof}Therefore, we obtain the recursion relation in \eqref{eq:detM_rec}, after comparing  \eqref{eq:detMk_fac} with  \eqref{eq:Omega_2k_1_u_rational}.
\subsection{Reduction of $ \det \mathbf N_k$ to $ \det\mathbf M_{k}$ and $ \det\mathbf N_{k-1}$}
Before factorizing $\omega_{2k} $ (as generalizations of Propositions \ref{prop:W_N2_1} and \ref{prop:W_N2_0}), we need to build some asymptotic formulae on hypergeometric techniques.
\begin{lemma}[Euler--Gau{\ss}--Schafheitlin--Weber]We have  \begin{align}\int_0^\infty I_0(\sqrt{u}t)K_0(t)t^n\D t={}&\begin{cases}-\dfrac{\log(1-u)}{2}+O(1), & n=0, \\
\dfrac{2^{n-1}(n-1)!}{(1-u)^{n}}+o\left(\dfrac{1}{(1-u)^{n}}\right), & n\in\mathbb Z_{>0,} \\
\end{cases}\label{eq:I0K0}\intertext{and}
\int_0^\infty I_1(\sqrt{u}t)K_0(t)t^{n}\D t={}&\begin{cases}-\dfrac{\log(1-u)}{2}+O(1), & n=0, \\
\dfrac{2^{n-1}(n-1)!}{(1-u)^{n}}+o\left(\dfrac{1}{(1-u)^{n}}\right), & n\in\mathbb Z_{>0,} \\
\end{cases}\label{eq:I1K0}
\end{align}as $ u\to1^-$.\end{lemma}\begin{proof}According to  the  modified Weber--Schafheitlin integral formula \cite[][\S13.45]{Watson1944Bessel}, we have\begin{align}\int_0^\infty I_0(\sqrt{u}t)K_0(t)t^n\D t={}&2^{n -1} \left[\Gamma \left(\frac{n +1}{2}\right)\right]^2 \, _2F_1\left(\left.\begin{array}{c}
\frac{n+1}{2},\frac{n+1}{2} \\
1 \\
\end{array}\right|u\right),\label{}\\\int_0^\infty I_1(\sqrt{u}t)K_0(t)t^n\D t={}&2^{n -1} \sqrt{u}\left[\Gamma \left(\frac{n +2}{2}\right)\right]^2 \, _2F_1\left(\left.\begin{array}{c}
\frac{n+2}{2},\frac{n+2}{2} \\
2 \\
\end{array}\right|u\right),\end{align}where the $_2F_1$'s are hypergeometric functions. When $ n=0$, the asymptotic behavior $ -\frac{\log(1-u)}{2}+O(1)$ can be found directly in both cases above; to prove \eqref{eq:I0K0} [resp.~\eqref{eq:I1K0}] when $n\in\mathbb Z_{>0}$, we need to specialize the Gau{\ss} summation \cite[][Theorem 2.2.2]{AAR}: \begin{align}
_2F_1\left(\left.\begin{array}{c}
a,b\ \\
c
\end{array}\right|1\right)=\frac{\Gamma(c)\Gamma(c-a-b)}{\Gamma(c-a)\Gamma(c-b)},\quad \text{for }\R(c-a-b)>0
\end{align}and the Euler transformation  \cite[][Theorem 2.2.5]{AAR}: \begin{align}
_2F_1\left(\left.\begin{array}{c}
a,b\ \\
c
\end{array}\right|u\right)=(1-u)^{c-a-b}{_2F_1}\left(\left.\begin{array}{c}
c-a,c-b\ \\
c
\end{array}\right|u\right)
\end{align}to $a=\frac{1-n}{2} ,b=\frac{1-n}{2},c=1$ (resp.~$a=\frac{2-n}{2} ,b=\frac{2-n}{2},c=2$).\end{proof}\begin{proposition}[Factorization of $ \omega_{2k}(1^{-})$]We have the following identity:\begin{align}
\lim_{u\to1^-}(1-u)^k\omega_{2k}(u)=(-1)^{\frac{k(k-1)}{2}}\frac{(k-1)!}{2^{(2k-1)k+1}}\det \mathbf N_{k-1}\det\mathbf N_{k}.\label{eq:W_detNk_fac}
\end{align}\end{proposition}
\begin{proof}We will use the representation of $ (2\sqrt{u})^{(2k-1)k}\omega_{2k}(u)$ in \eqref{eq:omega_2k_alg}.

From the exponential decays (for large $t$) in the respective integrands, it is clear that the following limits exist as finite real numbers, so long as $ j\in[1,k]\cup[k+2,2k]$ and $\ell\in\mathbb Z\cap[1,k] $:\begin{align}
\lim_{u\to1^-}\nu^\ell_{k,j}(u)={}&\nu^\ell_{k,j}(1),\\\lim_{u\to1^-}\acute\nu^\ell_{k,j}(u)={}&\acute\nu^\ell_{k,j}(1).
\end{align}So we need to examine the behavior of $ (1-u)^k\nu^\ell_{k,k+1}(u)$ and $ (1-u)^k\acute\nu^\ell_{k,k+1}(u)$, as $u$ approaches $1$ from below.

First, we consider \begin{align}
\nu^\ell_{k,k+1}(u)=\int_{0}^\infty I_0(\sqrt{u}t)[I_0(t)]^k[K_0(t)]^{k+1}t^{2\ell-1}\D t.
\end{align}When $ 2\ell-k-1<0$, the integral $ \nu^{\ell}_{k,k+1}(1)$ is finite (thanks to power law decay of the integrand for large $t$), and is equal to $ \lim_{u\to1^-}\nu^{\ell}_{k,k+1}(u)$.
Using the fact that \begin{align}
\sup_{t>0}t^{k}\left\vert[I_0(t)K_0(t)]^k-\frac{1}{(2t)^{k}}\right\vert<\infty,\label{eq:IKk_sup}
\end{align}we may deduce \begin{align}\begin{split}
\nu^\ell_{k,k+1}(u)={}&\frac{1}{2^{k}}\int_{0}^\infty I_0(\sqrt{u}t)K_0(t)t^{2\ell-1-k}\D t\\&+\int_{0}^\infty I_0(\sqrt{u}t)K_0(t)\left\{ [I_0(t)K_0(t)]^k-\frac{1}{(2t)^{k}} \right\}t^{2\ell-1}\D t\\={}&O\left( \int_{0}^\infty I_0(\sqrt{u}t)K_0(t)t^{2\ell-1-k}\D t \right)=\begin{cases}O((1-u)^{k+1-2\ell}), & 2 \ell>k+1\\
O(\log(1-u)), & 2 \ell=k+1 \\
\end{cases}\end{split}\label{eq:IvKM11_est}
\end{align}when $ 2\ell-1-k\in\mathbb Z_{\geq0}$, and \eqref{eq:I0K0} is applicable.

Then, we consider\begin{align}\acute
\nu^\ell_{k,k+1}(u)=\int_{0}^\infty I_1(\sqrt{u}t)[I_0(t)]^k[K_0(t)]^{k+1}t^{2\ell}\D t.
\end{align}When $ 2\ell-k<0$, the integral $\acute \nu^{\ell}_{k,k+1}(1)$ is finite (thanks to power law decay of the integrand for large $t$), and is equal to $ \lim_{u\to1^-}\acute\nu^{\ell}_{k,k+1}(u)$.
Using   \eqref{eq:IKk_sup} and \eqref{eq:I1K0}, we may deduce \begin{align}\begin{split}
\acute\nu^\ell_{k,k+1}(u)={}&\frac{1}{2^{k}}\int_{0}^\infty I_1(\sqrt{u}t)K_0(t)t^{2\ell-k}\D t\\&+\int_{0}^\infty I_1(\sqrt{u}t)K_0(t)\left\{ [I_0(t)K_0(t)]^k-\frac{1}{(2t)^{k}} \right\}t^{2\ell}\D t\\={}&O\left( \int_{0}^\infty I_1(\sqrt{u}t)K_0(t)t^{2\ell-k}\D t \right)=\begin{cases}O((1-u)^{k-2\ell}), & 2 \ell>k\\
O(\log(1-u)), & 2 \ell=k \\
\end{cases}\end{split}\label{eq:IvKM11_est}
\end{align}when $ 2\ell-k\in\mathbb Z_{\geq0}$.

Summarizing the efforts in the last two paragraphs, we see that only the term $ (1-u)^k\acute\nu^k_{k,k+1}(u)$ will  play a consequential r\^ole in the $ u\to1^-$ regime. Applying the bound\begin{align}
\sup_{t>0}t^{k+1}\left\vert[I_0(t)K_0(t)]^k-\frac{1}{(2t)^{k}}\right\vert<\infty
\end{align}to\begin{align}\begin{split}\acute
\nu^k_{k,k+1}(u)={}&\frac{1}{2^{k}}\int_{0}^\infty I_1(\sqrt{u}t)K_0(t)t^{k}\D t\\&+\int_{0}^\infty I_1(\sqrt{u}t)K_0(t)\left\{ [I_0(t)K_0(t)]^k-\frac{1}{(2t)^{k}} \right\}t^{2k}\D t,\end{split}
\end{align} we have\begin{align}\lim_{u\to1^-}
(1-u)^k\acute\nu^k_{k,k+1}(u)=\lim_{u\to1^-}\frac{(1-u)^k}{2^{k}}\int_{0}^\infty I_1(\sqrt{u}t)K_0(t)t^{k}\D t=\frac{(k-1)!}{2}
\end{align} according to \eqref{eq:I1K0}.

As we perform cofactor expansion  with respect to the matrix element $\lim_{u\to1^-}
(1-u)^k\acute\nu^k_{k,k+1}(u)$, manipulate columns according to \begin{align}
\left\{ \begin{array}{l}
\nu^\ell_{k,j}(1)=\nu^\ell_{k,k+j}(1)=\nu^\ell_{k,j}, \\\acute
\nu^\ell_{k,k+j}(1)-\acute \nu^\ell_{k,j}(1)=-\nu_{k-1,j-1}^\ell
\end{array} \right.
\end{align}for all $  j\in\mathbb Z\cap[2,k]$, and permute rows for a total of $\sum_{j=1}^k[(k+j)-2j]=\frac{k(k-1)}{2} $ times (according to bubble sort), we can identify $2^{(2k-1)k}\lim_{u\to1^-}(1-u)^k\omega_{2k}(u) $    with\begin{align}\begin{split}&
(-1)^{k+1+\frac{k(k-1)}{2}}\frac{(k-1)!}{2}\det\left( \begin{array}{c|c}\begin{array}{ccc}
 &  &  \\
 & \raisebox{-0.25\height}{\resizebox{1.75\width}{1.75\height}{$\mathbf N_k^{\mathrm T}$}} &  \\
 &  &  \\
\end{array}&\raisebox{-0.25\height}{\resizebox{1.75\width}{1.75\height}{$\mathbf O_{}$}}\\\hline\begin{array}{ccc}\vspace{-0.75em}\\\acute
\nu_{k,1}^1(1) & \cdots&\acute\nu^1_{k,k}(1) \\
\multicolumn{3}{c}{\cdots\cdots\cdots\cdots\cdots\cdots\cdots}   \\
\acute
\nu_{k,1}^{k-1}(1) & \cdots&\acute\nu^{k-1}_{k,k}(1) \\
\end{array}&\raisebox{-0.25\height}{\resizebox{1.75\width}{1.75\height}{$-\mathbf N_{k-1}^{\mathrm T}$}}\end{array} \right)\\={}&(-1)^{\frac{k(k-1)}{2}}\frac{(k-1)!}{2}\det \mathbf N_{k-1}\det\mathbf N_{k},\end{split}
\end{align}  as expected. \end{proof}

\begin{proposition}[Factorization of $ \omega_{2k}(0^{+})$]The limit\begin{align}
\lim_{u\to0^+}u^{k^2}\omega_{2k}(u)=(-1)^{\frac{k(k-1)}{2}}\frac{(2k+1)(\det \mathbf M_k)^{2}}{2^{(2k-1)k+1}(k+1)}
\end{align} entails\begin{align}
\omega_{2k}(u)=(-1)^{\frac{k(k-1)}{2}}\frac{(2k+1)(\det \mathbf M_k)^{2}}{2^{(2k-1)k+1}u^{k^{2}}(k+1)}\prod _{j=1}^{k+1}\left[\frac{(2j-1)^{2}}{(2j-1)^2-u}\right]^{k},\quad \forall u\in(0,1).\label{eq:omega_2k_u_rational}
\end{align} \end{proposition}
\begin{proof}In the formula\begin{align}
2^{(2k-1)k}u^{k^{2}}\omega_{2k}(u)={}&\det\left(\begin{array}{rrr}\nu_{k,1}^1(u) & \cdots\ & \nu_{k,2k-1}^1(u) \\
\sqrt{u}\acute\nu_{k,1}^1(u) & \cdots\ & \sqrt{u}\acute \nu_{k,2k-1}^1(u) \\
\multicolumn{3}{c}{\cdots\cdots\cdots\cdots\cdots\cdots\cdots\cdots\cdots}    \\
\nu_{k,1}^k(u) & \cdots\ & \nu_{k,2k-1}^k(u)  \\
\sqrt{u}\acute
\nu_{k,1}^k(u) & \cdots\ &\sqrt{u}\acute \nu_{k,2k-1}^k(u)\\
\end{array}\right),
\end{align}we observe that\begin{align}\nu_{k,j}^\ell(u)=
\begin{cases}O(\log u), & j\in\{1\}\cup(\mathbb Z\cap[k+2,2k]) \\
\mu^\ell_{k-1,j-1}+O(u), & j\in\mathbb Z\cap[2,k]
\end{cases}\end{align}and \begin{align}
\sqrt{u}\acute \nu_{k,j}^\ell(u)=\begin{cases}-\frac{2k+1}{2(k+1)}\mu^\ell_{k-1,1}+o(1), & j=1 \\
O(u), & j\in\mathbb Z\cap[2,k] \\
-\mu^{\ell}_{k-1,j-k-1}+o(1), & j\in\mathbb Z\cap[k+2,2k] \\
\end{cases}
\end{align}apply to all $ \ell\in\mathbb Z\cap[1,k]$, in the $ u\to0^+$ limit. The factorization procedure is thus a straightforward  generalization of Proposition  \ref{prop:W_N2_0}.  \end{proof}Comparing \eqref{eq:W_detNk_fac} with \eqref{eq:omega_2k_u_rational}, we arrive at \eqref{eq:detN_rec}, thereby completing the proof of Broadhurst--Mellit determinant formulae (Conjectures \ref{conj:BMdetM} and \ref{conj:BMdetN}).

\section{Vacuum diagrams and Mahler measures\label{sec:VacMahler}}

So far, each Wro\'nskian in our derivations  concerns a set of functions  that all reside in the kernel space   $\ker\widetilde L_n $ of a certain  Vanhove operator $ \widetilde L_n$. The proofs of both Conjectures \ref{conj:BMdetM} and \ref{conj:BMdetN} were built on homogeneous evolution equations for the corresponding Wro\'nskian determinants, namely, \eqref{eq:W_ODE_Mk_prep0}  and \eqref{eq:W_ODE_Nk_prep0}. In this section, we will treat a pair of two-scale vacuum diagrams that are not annihilated by Vanhove's operators, along with  the corresponding ``vacuum analogs'' $ \check \varOmega_3(u)$ and $ \check\omega_4(u)$ of the  Wro\'nskian determinants $ \varOmega_3(u)$ and $ \omega_4(u)$ factorized in \S\S\ref{sec:detM2}--\ref{sec:detN2}. The inhomogeneous evolution equations for these new  Wro\'nskians  $ \check \varOmega_3(u)$ and $ \check\omega_4(u)$ eventually enable us to verify Theorem \ref{thm:BM_Mahler}, through factorizations of determinants.
\subsection{Conjectures of Broadhurst--Mellit and Rodr\'iguez-Villegas} For each positive integer $n$, the following integral\begin{align}
V_{n}:=\IKM(0,n;1)=\int_0^\infty[K_0(t)]^n t\D t,
\end{align}is known as  the $(n-1)$-loop vacuum diagram \cite[][(1)]{BBBG2008} in two-dimensional quantum field theory.  An integral representation $K_0(t):=\int_0^\infty e^{-t\cosh u}\D u,t>0$ connects $V_n$ to its avatar in statistical mechanics:\begin{align}
V_n=\int_0^\infty \D x_1\cdots\int_0^\infty \D x_n\frac{1}{(\cosh x_1+\cdots +\cosh x_n)^2},
\end{align} which is called  the $n$th integral of Ising class \cite{BBC2006Ising,BBBC2007Ising}.
It has been shown that \cite{Ouvry2005,BBBG2008} \begin{align}
V_{1}=1,\quad V_2=\frac{1}{2},\quad V_3=\frac{3}{4}\sum_{n=0}^\infty\left[ \frac{1}{(3n+1)^2}-\frac{1}{(3n+2)^2} \right],\quad V_4=\sum_{n=0}^\infty\frac{1}{(2n+1)^3}
\end{align}and \cite[][Theorem 2]{BBC2006Ising}\begin{align} \lim_{n\to\infty}\frac{2^nV_n}{n!}=2e^{-2\gamma},\end{align}where  $ \gamma:=\lim_{n\to\infty}\left(-\log n +\sum_{k=1}^n\frac1k\right)$ is the Euler--Mascheroni constant. The intermediate regime (namely, vacuum diagrams $ V_{n}$ for $n\in\mathbb Z_{>4}$) appears to be an uncharted territory.

In 2013, Broadhurst wrote that ``we know nothing about the number theory of $ V_5$'' \cite[][\S8.6]{Broadhurst2013MZV}, which stood in stark contrast with  other physically relevant Bessel moments $ \IKM(a,b;2k+1)$  involving $ a+b=5$ Bessel factors,   where $ k$ is a non-negative integer.
In particular, conjectures on the closed-form expressions of $ \IKM(1,4;2k+1)$ and $ \IKM(2,3;2k+1)$ for $ k\in\mathbb Z_{\geq0}$ have been supported by numerical experiments \cite{BBBG2008} and confirmed by theoretical analyses \cite{BBBG2008,BlochKerrVanhove2015,Samart2016,Zhou2017WEF}.

Rising to the challenge of understanding $ V_5=\IKM(0,5;1)$ and $ V_6=\IKM(0,6;1)$ arithmetically, Broadhurst and Mellit \cite{BroadhurstMellit2016,Broadhurst2016} have proposed a possible link between Bessel moments and  special $L$-values attached to  two special modular forms\begin{align}
f_{3,15}(z)={}&[\eta(3z)\eta(5z)]^3+[\eta(z)\eta(15z)]^3,\\f_{4,6}(z)={}&[\eta(z)\eta(2z)\eta(3z)\eta(6z)]^{2},
\end{align}with $ \eta(z):=e^{\pi iz/12}\prod_{n=1}^\infty(1-e^{2\pi inz})$ being the Dedekind eta function defined for complex numbers $z$ in the upper half-plane $\mathfrak H:=\{w\in\mathbb C| \I w>0\}$.  Here,  $f_{k,N} $ represents a modular form of weight $k$ and level $N$.

We recapitulate their conjectures (see \cite[][(4.3), (5.8)]{BroadhurstMellit2016} or \cite[][(101), (114)]{Broadhurst2016})
below.
\begin{conjecture}[Broadhurst--Mellit]\label{conj:L5(4)L6(5)BM}We have the following evaluation of two $ 2\times 2$ determinants filled with Bessel moments:\begin{align}\det \check {\mathbf M}_2:=\det
\begin{pmatrix}\IKM(0,5;1) & \IKM(0,5;3) \\
\IKM(2,3;1) & \IKM(2,3;3) \\
\end{pmatrix}\overset{?}{=}{}&\frac{45}{8\pi^{2}}L(f_{3,15},4),\label{eq:L5(4)}\\\det \check {\mathbf N}_2:=\det
\begin{pmatrix}\IKM(0,6;1) & \IKM(0,6;3) \\
\IKM(2,4;1) & \IKM(2,4;3) \\
\end{pmatrix}\overset{?}{=}{}&\frac{27}{4\pi^2}L(f_{4,6},5),\label{eq:L6(5)}
\end{align}where \begin{align}
L(f_{k,N},s):=\frac{(2\pi )^{s}}{\Gamma(s)}\int_0^{\infty} f_{k,N}(iy)y^{s-1}\D y.
\end{align}\end{conjecture}
In his seminal work   \cite[][\S7.4]{Broadhurst2016}, Broadhurst has observed    intricate connections between vacuum diagrams and  logarithmic Mahler measures $ m(P)$ of Laurent polynomials $ P\in\mathbb C[x_1^{\pm1},\dots, x_n^{\pm1}]$ [cf.~\eqref{eq:defn_Mahler_m}].
Proven results in vacuum diagrams \cite{Ouvry2005,BBBG2008} and Mahler measures \cite{Boyd1981b} bring us the following identities \cite[][(118) and (119)]{Broadhurst2016}:\begin{align}
V_3=\frac{\pi}{\sqrt{3}}m(1+x_1+x_2),\quad V_4=\frac{\pi^{2}}{4}m(1+x_{1}+x_{2}+x_{3}).\label{eq:V3V4}
\end{align} Intriguingly, the special values $ L(f_{3,15},4)$ and $L(f_{4,6},5) $ defined in \eqref{eq:L5(4)} and \eqref{eq:L6(5)} also show up in the conjectural evaluations of two logarithmic Mahler measures, due to Fernando Rodr\'{\i}guez-Villegas (see  \cite[][\S8]{BoydLindRVDeninger2003}, \cite[][(6.11), (6.12)]{BSWZ2012} and \cite[][(120), (121)]{Broadhurst2016}).
\begin{conjecture}[Rodr\'{\i}guez-Villegas]\label{conj:L5(4)L6(5)Mahler}We have\begin{align}
m(1+x_{1}+x_{2}+x_{3}+x_{4})\overset{?}{=}{}&6\left( \frac{\sqrt{15}}{2\pi} \right)^{5}L(f_{3,15},4),\\m(1+x_{1}+x_{2}+x_{3}+x_{4}+x_{5})\overset{?}={}&3\left( \frac{\sqrt{6}}{\pi} \right)^6L(f_{4,6},5).
\end{align}\end{conjecture}  It appears that neither Conjecture \ref{conj:L5(4)L6(5)BM} nor \ref{conj:L5(4)L6(5)Mahler} would yield to the algebraic methods developed in this paper. In a  recent review  \cite{StraubZudilin2018}, Straub and Zudilin have stated that Conjecture  \ref{conj:L5(4)L6(5)Mahler} remains unproven, as of January 2018. Nevertheless, we can still achieve a modest goal of demonstrating the equivalence between  Conjectures \ref{conj:L5(4)L6(5)BM} and \ref{conj:L5(4)L6(5)Mahler}, as stated in Theorem \ref{thm:BM_Mahler}.

As we will witness in the rest of \S\ref{sec:VacMahler}, the bridge that connects Bessel moments to Mahler measures is   Broadhurst's key formula    (see \cite[][(9)]{Broadhurst2009}, \cite[][last formula on p.~978 and penultimate formula on p.~981]{BSWZ2012}, as well as  \cite[][(122)]{Broadhurst2016}):
\begin{align}
m(1+x_{1}+\cdots+x_{n-1})=-\gamma+\log 2-n\int_0^\infty J_{1}(t)[J_0(t)]^{n-1}\log t\D t,\label{eq:BroadhurstMahler}
\end{align}which is provable by differentiating the ``ramble integral'' (see   \cite[][\S6]{BSWZ2012} and \cite[][(2--2)]{BSW2013})\begin{align}W_n(s):={}&
\int_0^1\D t_1\cdots\int_0^1\D t_n\left\vert \sum_{k=1}^n e^{2\pi i t_k} \right\vert^s\notag\\={}&-2^{s}\frac{\Gamma\left( 1+\frac{s}{2} \right)}{\Gamma\left( 1-\frac{s}{2} \right)}\int_0^\infty x^{-s}\frac{\D}{\D x}[J_0(x)]^n\D x,\quad \forall s\in(-n/2,2)\label{eq:zeta_Mahler}
\end{align}at $ s=0^{}$.
Here, we remind our readers that  $ J_0(x):=\frac{2}{\pi}\int_0^{\pi/2}\cos(x\cos\varphi)\D\varphi$ is the Bessel function of the first kind and\ zeroth order, whose derivative gives $ \D J_0(x)/\D x=-J_1(x)$.

     \subsection{Relation between $ \det \check {\mathbf M}_2$ and $ m(1+x_1+x_2+x_3+x_4)$\label{subsec:red_V5}} If we assign a different parameter to one of the internal lines in the diagram $V_{5}$, then we obtain a family of two-scale vacuum diagrams\begin{align}
\int_{0}^{\infty}K_0(\sqrt{u}t)[K_0(t)]^4t\D t
\end{align} parametrized by $ u>0$.
To study this family of two-scale diagrams, we need a modest extension to Lemma \ref{lm:VanhoveLn}, as given below.
\begin{proposition}[Differential equation for  two-scale 4-loop vacuums]\label{lm:Vanhove_L3_ODE}We have \begin{align}\widetilde
L_3\IKvM(0,5;1|u)=\frac{3}{2}\log u, \quad \forall u\in(0,\infty),\label{eq:PF_Vv5_u_all}
\end{align}where $\widetilde {L}_3$ is the third-order Vanhove operator defined in \eqref{eq:VL3}. \end{proposition}\begin{proof}
We first note that \begin{align}\begin{split}
 \widetilde {L}_3K_0(\sqrt{u}t)={}&\frac{ (2u^{2} -25u+32)t^2+2 (u-4) }{2} K_0( \sqrt{u}t)\\{}&-\frac{  [ (u-16) (u-4)t^2+12 (u-6)]\sqrt{u}t }{8} K_1( \sqrt{u}t),\end{split}\label{eq:A4_K0}
\end{align}where $ K_1(x)=-\D K_0(x)/\D x$, which specializes to \begin{align}
\widetilde L_3\IKvM(0,5;1|1)=\frac{3}{8} \int _{0}^\infty  [4 (3 t^2-2) K_0(t)+5 t (4-3 t^2) K_1(t)][K_0(t)]^4t\D t=0.\label{eq:PF_Vv5_u0}
\end{align} Here, we have canceled out  integrals in the last step, thanks to   the following formula  for $ n\in\mathbb Z_{>0}$:  \begin{align}
\int_{0}^{\infty}K_{1}(t)[K_0(t)]^4t^{2n}\D t={}&\frac{2n}{5}\IKM(0,5;2n-1),\label{eq:K1KKKK}
\end{align} which is a consequence of integration by parts.

We have\begin{align}
t\widetilde {L}_3K_0(\sqrt{u}t)=-\frac{1}{2^{3}} L _{5}^*\frac{K_0(\sqrt{u}t)}{t},\label{eq:Vanhove_L_adj}
\end{align}where \begin{align}\begin{split}
L^*_5:={}&-t^5 \frac{\partial^{5} }{\partial t^{5}}-15t^{4} \frac{\partial^{4} }{\partial t^{4}}+5 t^3 (4 t^2-13)\frac{\partial^{3} }{\partial t^{3}}+90 t^2 (2 t^2-1)\frac{\partial^{2} }{\partial t^{2}}\\{}&-t (64 t^4-392 t^2+31)\frac{\partial }{\partial t}-(192 t^4-184 t^2+1).\end{split}
\end{align}
Here, 
 the differential operator $ L _{5}^*$ is (formally) adjoint to the Borwein--Salvy operator  \cite[][Example 4.1]{BorweinSalvy2007}\begin{align}\begin{split}
L_5:={}&\eth^5-20t^2\eth^3-60t^{2}\eth^2+8t^2(8t^2-9)t\eth^1+32t^2(4t^2-1)\eth^0\\{}&\left[\text{where }\eth^{n}:=\left(t\frac{\partial }{\partial t}\right)^n\right],
\end{split}\end{align}an annihilator of every member in the set $ \{[I_0(t)]^j[K_0(t)]^{4-j}|j\in[0,4]\}$.

Using the fact that $ L_5\{[K_0(t)]^4\}=0$,  the recursive construction of $ L_5=\mathscr L_{5,5}$ via the the Bron\-stein--Mulders--Weil algorithm \cite[][Theorem 1]{BMW1997}:\begin{align}
\begin{cases}\mathscr L_{5,0}=\eth^0,\mathscr L_{5,1}=\eth^1, &  \\
\mathscr L_{5,k+1}=\eth^1\mathscr L_{5,k}-k(5-k)t^{2}\mathscr L_{5,k-1}, &          \forall k\in\mathbb Z\cap[1,4], \\
\end{cases}\label{eq:BMW_5}
\end{align}along with the identities $ \mathscr L_{5,k}\{[K_0(t)]^4\}=\frac{4!}{(4-k)!}[K_0(t)]^{4-k}[\eth^1 K_0(t)]^{k} ,\forall k\in\mathbb Z\cap[1,4]$ \cite[][Lemma 3.1]{BorweinSalvy2007},
we can integrate by parts as follows:\begin{align}\begin{split}
0={}&\int_{0}^{\infty}\frac{K_0(\sqrt{u}t)-K_{0}(t)}{t}L_5\{[K_0(t)]^4\}\D t\\={}&\int_{0}^{\infty}[K_0(\sqrt{u}t)-K_{0}(t)]\frac{\partial}{\partial t}\mathscr L_{5,4}\{[K_0(t)]^4\}\D t-4\int_{0}^{\infty}t[K_0(\sqrt{u}t)-K_{0}(t)]\mathscr L_{5,3}\{[K_0(t)]^4\}\D t\\={}&24\log\sqrt{u}-\int_{0}^{\infty}\mathscr L_{5,4}\{[K_0(t)]^4\}\frac{\partial[K_0(\sqrt{u}t)-K_{0}(t)]}{\partial t}\D t-4\int_{0}^{\infty}t[K_0(\sqrt{u}t)-K_{0}(t)]\mathscr L_{5,3}\{[K_0(t)]^4\}\D t\\={}&12\log  u+\int_{0}^{\infty}[K_0(t)]^4L_5^{*}\frac{K_0(\sqrt{u}t)-K_{0}(t)}{t}\D t.\end{split}\label{eq:Vv5_int_part}
\end{align}
Here, in the first step of integration by parts, the boundary contribution arises from the asymptotic behavior $K_0(\sqrt{u}t)-K_{0}(t)=-\log\sqrt{u}+O(t^{2}\log t), t\to0^+$; all the subsequent transfers of derivatives involve no boundary terms at all. Recalling \eqref{eq:PF_Vv5_u0} and \eqref{eq:Vanhove_L_adj}, we see that \eqref{eq:Vv5_int_part} brings us  \eqref{eq:PF_Vv5_u_all}.
   \end{proof}     \begin{remark}As we specialize the relation\begin{align}D^{1}
 \int_{0}^{\infty}[K_0(t)]^4t\widetilde L_3K_0(\sqrt{u}t)\D t=\frac{3}{2}D^{1 }\log u
\end{align}to $u=1$, we obtain\begin{align}
\IKM(0,5;5)=\frac{76}{15}\IKM(0,5;3)-\frac{16}{45}\IKM(0,5;1)+\frac{8}{15},\label{eq:IKM055}
\end{align} a relation that was previously conjectured in \cite[][(120)]{BBBG2008}.      \eor\end{remark}

 We will be interested in a $ 3\times3$ determinant \begin{align} \check \varOmega_3(u):=W[\IKvM(0,5;1|u),\IvKM(2,3;1|u),\IKvM(2,3;1|u)],\end{align} which is a ``vacuum analog'' of another Wro\'nskian studied in  \S\ref{sec:detM2}:\begin{align} \varOmega_3(u):=W\left[\frac{\IvKM(1,4;1|u)+4\IKvM(1,4;1|u)}{5},\IvKM(2,3;1|u),\IKvM(2,3;1|u)\right].\end{align}

\begin{lemma}[Differential equation for $\check \varOmega_3(u)$]For $ u\in(0,4)$, we have \begin{align}\begin{split}
D^{1}\check \varOmega_3(u)={}&\frac{3\check\varOmega_3(u)}{2}D^1\log\frac{1}{u^2 (4-u) ( 16-u)}\\{}&+\frac{3}{2}\frac{\log u}{u^2 (4-u) ( 16-u)}\det\begin{pmatrix}D^0\mu^1_{2,2}(u) & D^0\mu^1_{2,3}(u) \\
D^1\mu^1_{2,2}(u) & D^1\mu^1_{2,3}(u) \\
\end{pmatrix},\end{split}\label{eq:check_Omega3_ODE}
\end{align}where $\mu^1_{2,2}(u)= \IvKM(2,3;1|u)$ and $ \mu^1_{2,3}(u)=\IKvM(2,3;1|u)$.\end{lemma}\begin{proof}Differentiating each row of the Wro\'nskian determinant $ \check \varOmega_3(u)$, we obtain\begin{align}
D^{1}\check \varOmega_3(u)=\det\begin{pmatrix}D^0\IKvM(0,5;1|u) & D^0\mu^1_{2,2}(u) & D^0\mu^1_{2,3}(u) \\
D^1\IKvM(0,5;1|u) & D^1\mu^1_{2,2}(u) & D^1\mu^1_{2,3}(u) \\
D^3\IKvM(0,5;1|u) & D^3\mu^1_{2,2}(u) & D^3\mu^1_{2,3}(u) \\
\end{pmatrix}.
\end{align} Using the differential equations in  \eqref{eq:PF_Vv5_u_all}  to reduce the third-order derivatives to linear combinations of lower-order derivatives, we may convert the equation above into\begin{align}\begin{split}
D^{1}\check \varOmega_3(u)={}&\frac{3\check\varOmega_3(u)}{2}D^1\log\frac{1}{u^2 (4-u) ( 16-u)}\\{}&+\det\begin{pmatrix}D^0\IKvM(0,5;1|u) & D^0\mu^1_{2,2}(u) & D^0\mu^1_{2,3}(u) \\
D^1\IKvM(0,5;1|u) & D^1\mu^1_{2,2}(u) & D^1\mu^1_{2,3}(u) \\
\frac{3\log u}{2u^2 (4-u) ( 16-u)} & 0 & 0 \\
\end{pmatrix},\end{split}
\end{align} which is equivalent to the claimed identity. \end{proof}\begin{proposition}[An integral representation for $\check\varOmega_3(u)$]\label{prop:check_C3}The $2\times2 $ determinant appearing in \eqref{eq:check_Omega3_ODE} has an integral representation for $ u\in(0,4)$:\begin{align}
\det\begin{pmatrix}D^0\mu^1_{2,2}(u) & D^0\mu^1_{2,3}(u) \\
D^1\mu^1_{2,2}(u) & D^1\mu^1_{2,3}(u) \\
\end{pmatrix}={}&-\frac{\pi^4}{24}\frac{1}{\sqrt{u^2 (4-u) ( 16-u)}}\int_0^\infty J_0(\sqrt{u}t)[J_0(t)]^4t\D t.\label{eq:Sigma_2_int_repn}
\end{align}As a result, there exists a constant $ \check C_3\in\mathbb R$ such that \begin{align}\begin{split}
[u^2 (4-u) ( 16-u)]^{3/2}\check\varOmega_3(u)={}&\check C_3-\frac{\pi^{4}\sqrt{u}\log u}{8} \int_0^\infty J_1(\sqrt{u}t)[J_0(t)]^4\D t\\{}&+\frac{\pi^{4}}{4}\int_{0}^\infty \frac{1-J_0(\sqrt{u}t)}{t}[J_0(t)]^4\D t\end{split}\label{eq:check_C3}
\end{align}for $u\in(0,4) $.\end{proposition}\begin{proof}

By direct computation,  one can show that \begin{align}
\widetilde L_3\left[ \sqrt{u^2 (4-u) ( 16-u)} \det\begin{pmatrix}D^{0}f_{1}(u) & D^{0}f_2(u) \\
D^{1}f_{1}(u) & D^{1}f_2(u) \\
\end{pmatrix}\right]=0
\end{align} holds for any two functions $ f_1,f_2\in\ker\widetilde L_3$ that are annihilated by $ \widetilde L_3$. Therefore, for $ u\in(0,4)$,\begin{align}\varPsi_2(u):=
\sqrt{u^2 (4-u) ( 16-u)}\det\begin{pmatrix}D^0\mu^1_{2,2}(u) & D^0\mu^1_{2,3}(u) \\
D^1\mu^1_{2,2}(u) & D^1\mu^1_{2,3}(u) \\
\end{pmatrix}
\end{align} is a linear combination of\begin{align}
\frac{\IvKM(1,4;1|u)+4\IKvM(1,4;1|u)}{5},\;\;\IvKM(2,3;1|u),\text{ and } \IKvM(2,3;1|u),
\end{align} in view of \S\ref{sec:detM2}. However,
we can infer from \cite[][Propositions 3.1.2 and 5.1.4]{Zhou2017WEF} that \begin{align}\begin{split}&
\IvKM(1,4;1|u)+4\IKvM(1,4;1|u)\\={}&\frac{\pi^{4}}{6}\frac{p_{4}(\sqrt{u})}{\sqrt{u}}:=\frac{\pi^{4}}{6}\int_0^\infty J_0(\sqrt{u}t)[J_0(t)]^4t\D t
\end{split}\end{align}holds for $ u\in(0,4)$, so we have \begin{align}
\varPsi_2(u)=c_{1}\frac{p_{4}(\sqrt{u})}{\sqrt{u}}+c_2\IvKM(2,3;1|u)+c_{3}\IKvM(2,3;1|u),\quad \forall u\in(0,4),
\end{align}where the constants $ c_1,c_2,c_3$ will be determined from the asymptotic behavior of $\varPsi_2(u) $  in the $ u\to0^+$ limit and the special value $ \varPsi_2(1)$.

    We note that in the decomposition\begin{align}\begin{split}
\IKvM(2,3;1|u)={}&\frac12\int_0^\infty K_0(\sqrt{u}t)I_0(t)K_0(t)\D t\\{}&+\int_0^\infty K_0(\sqrt{u}t)I_0(t)K_0(t)\left[ I_0(t)K_0(t)-\frac{1}{2t} \right]t\D t,\end{split}
\end{align} we have \cite[cf.][(3.3)]{Bailey1936II}\begin{align}\begin{split}
\int_0^\infty K_0(\sqrt{u}t)I_0(t)K_0(t)\D t={}&\frac{\mathbf K\left( \sqrt{\frac{1}{2}\left( 1+i\sqrt{\frac{4-u}{u}} \right)} \right)\mathbf K\left( \sqrt{\frac{1}{2}\left( 1-i\sqrt{\frac{4-u}{u}} \right)} \right)}{\sqrt{u}}\\={}&\frac{1}{2\sqrt{4-u}}\log^2\sqrt{\frac{4-u}{u}}+O\left( \log\frac{4-u}{u} \right),\quad\text{as } u\to0^+,\end{split}
\end{align} with $ \mathbf K(\sqrt{\lambda})=\int_0^{\pi/2}(1-\lambda\sin^2\theta)^{-1/2}\D\theta$, and\begin{align}\begin{split}&
\int_0^\infty K_0(\sqrt{u}t)I_0(t)K_0(t)\left[ I_0(t)K_0(t)-\frac{1}{2t} \right]t\D t\\={}&O\left( \int_0^\infty K_0(\sqrt{u}t)I_0(t)K_0(t)\frac{\D t}{\sqrt{t}} \right)=O\left( \int_0^\infty [1+|\log(\sqrt{u}t)|]I_0(t)K_0(t)\frac{\D t}{\sqrt{t}} \right)\\={}&O(\log u)\end{split}
\end{align}according to \begin{align}
\sup_{t>0}t^{3/2}\left\vert I_0(t)K_0(t)-\frac{1}{2t} \right\vert<\infty\quad \text{and}\quad\sup_{t>0}\frac{K_0(t)}{1+|\log t|}<\infty.
\end{align}Thus, we have \begin{align}
\IKvM(2,3;1|u)=\frac{\log^2 u}{32}+O(\log u),\quad\text{as } u\to0^+,
\end{align}and similarly,\begin{align}
uD^{1}\IKvM(2,3;1|u)=\frac{\log u}{16}+O(1),\quad\text{as } u\to0^+.
\end{align} Therefore, we have \begin{align}\begin{split}
\varPsi_2(u)={}&8[D^{0}\IvKM(2,3;1|u)][uD^{1}\IKvM(2,3;1|u)]+O(1)\\={}&8\IKM(1,3;1)\frac{\log u}{16}+O(1)=\frac{\pi^2}{32}\log u+O(1)\end{split}
\end{align}in the regime $ u\to0^+$. Meanwhile, we recall that $ \frac{p_4(\sqrt{u})}{\sqrt{u}}=-\frac{3\log u}{4\pi^2}+O(1)$ \cite[][Theorem 4.4]{BSWZ2012} and $ \IvKM(2,3;$ $1|u)=O(1)$ as $ u\to0^+$, so we must have\begin{align}
\varPsi_2(u)=-\frac{\pi^{4}}{24}\frac{p_4(\sqrt{u})}{\sqrt{u}}+c_2\IvKM(2,3;1|u),\quad u\in(0,4),\label{eq:Sigma_2_prep}
\end{align}for a certain constant $c_2$.

Bearing in mind that \begin{align}\begin{split}&
D^1\IKvM(2,3;1|1)-D^1\IvKM(2,3;1|1)\\={}&-\frac{1}{2}\int_0^\infty[I_0(t)K_1(t)+I_1(t)K_0(t)]I_0(t)[K_0(t)]^2t^2\D t\\={}&-\frac{1}{2}\int_0^\infty I_0(t)[K_0(t)]^2t\D t=-\frac{\pi}{6\sqrt{3}},\end{split}\label{eq:IK_reduction}
\end{align} we compute \begin{align}\begin{split}
\varPsi_2(1)={}&3\sqrt{5}\det\begin{pmatrix}\IKM(2,3;1) & 0 \\
D^1\IvKM(2,3;1|1) & -\frac{\pi}{6\sqrt{3}} \\
\end{pmatrix}\\={}&-\frac{\pi\sqrt{5}}{2\sqrt{3}}\IKM(2,3;1)=-\frac{\pi^{4}}{24}p_4(1),
\end{split}\end{align}where the last equality can be inferred from \cite[][Proposition 3.1.2]{Zhou2017WEF}. Therefore, we have $ c_2=0$ in \eqref{eq:Sigma_2_prep}, which allows us to confirm \eqref{eq:Sigma_2_int_repn}.

A  solution to  \eqref{eq:check_Omega3_ODE}, namely\begin{align}\begin{split}D^{1}\check \varOmega_3(u)={}&\frac{3\check\varOmega_3(u)}{2}D^1\log\frac{1}{u^2 (4-u) ( 16-u)}\\{}&-\frac{\pi^4}{16}\frac{\log u}{[u^2 (4-u) ( 16-u)]^{3/2}}\int_0^\infty J_0(\sqrt{u}t)[J_0(t)]^4t\D t,\end{split}
\tag{\ref{eq:check_Omega3_ODE}$'$}
\end{align}has the form\begin{align}\check
\varOmega_3(u)=\frac{1}{[u^2 (4-u) ( 16-u)]^{3/2}}\left( \check C_3- \frac{\pi^4}{16}\int_{0}^u\left\{ \int_0^\infty J_0(\sqrt{v}t)[J_0(t)]^4t\D t \right\}\log v\D v\right),
\end{align}where the constant of integration $ \check C_3$ is equal to $ 2^9\lim_{u\to0^+}u^3\check \varOmega_3(u)$.

Here, noting that \begin{align}
J_0(\sqrt{v}t)=\frac{\partial}{\partial v}\frac{2\sqrt{v}J_1(\sqrt{v}t)}{t},\quad\frac{tJ_1(\sqrt{v}t)}{2\sqrt{v}}=-\frac{\partial J_0(\sqrt{v}t)}{\partial v}
\end{align} we may integrate by parts, as follows:\begin{align}\begin{split}&
\int_{0}^u\left\{ \int_0^\infty J_0(\sqrt{v}t)[J_0(t)]^4t\D t \right\}\log v\D  v\\={}&(2\sqrt{u}\log u) \int_0^\infty J_1(\sqrt{u}t)[J_0(t)]^4\D t-\int_{0}^u\left\{ \int_0^\infty \frac{2J_1(\sqrt{v}t)}{\sqrt{v}}[J_0(t)]^4\D t \right\}\D  v\\={}&(2\sqrt{u}\log u) \int_0^\infty J_1(\sqrt{u}t)[J_0(t)]^4\D t-4 \int_0^\infty \frac{1-J_0(\sqrt{u}t)}{t}[J_0(t)]^4\D t.\end{split}
\end{align}This completes the proof of \eqref{eq:check_C3}.\end{proof}To facilitate computations of the Wro\'nskian matrix $ \check \varOmega_3(u)$, we recall the notations $ \IpKM$ and $ \IKpM$ from Definition \ref{defn:IpKM_IKpM}, before writing down the following analog of \eqref{eq:Omega3u_alt}:\begin{align}\begin{split}
2^{3}u^{3/2}\check\varOmega_3(u)={}&\det\begin{pmatrix}\IKvM(0,5;1|u)&\IvKM(2,3;1|u)&\IKvM(2,3;1|u) \\
\IKpM(0,5;1|u)&\IpKM(2,3;1|u)&\IKpM(2,3;1|u) \\
\IKvM(0,5;3|u)&\IvKM(2,3;3|u)&\IKvM(2,3;3|u) \\
\end{pmatrix}\\={}&\det\begin{pmatrix}\IKvM(0,5;1|u)&\mu^1_{2,2}(u) & \mu^1_{2,3}(u) \\
\IKpM(0,5;1|u)&\acute\mu^1_{2,2}(u) & \acute\mu^1_{2,3}(u) \\
\IKvM(0,5;3|u)&\mu^2_{2,2}(u) & \mu^2_{2,3}(u) \\
\end{pmatrix}.\end{split}\label{eq:IK_echelon}
\end{align}

In the next proposition, we factorize the last determinant in the $ u\to0^+$ regime.\begin{proposition}[Factorization of $ \check \varOmega_3(0^+)$]We have \begin{align}\check
C_{3}=2^9\lim_{u\to0^+}u^3\check \varOmega_3(u)=\pi^{2 }V_4.
\end{align}Consequently, we have \begin{align}
135\sqrt{5}\check\varOmega_3(1)=\pi^{2 }V_4+\frac{\pi^{4}}{4}\int_{0}^\infty \frac{1-J_0(t)}{t}[J_0(t)]^4\D t.\label{eq:check_Omega_3_1_int_repn}
\end{align}\end{proposition}
\begin{proof}Using methods in Proposition \ref{prop:W_M2_0}, we can  show that \begin{align}\begin{split}
2^3u^{3}\check\varOmega_3(u)={}&\det\left(\begin{array}{rrr}\IKvM(0,5;1|u)&\IvKM(2,3;1|u)&\IKvM(2,3;1|u) \\
\sqrt{u}\IKpM(0,5;1|u)&\sqrt{u}\IpKM(2,3;1|u)&\sqrt{u}\IKpM(2,3;1|u) \\
u\IKvM(0,5;3|u)&u\IvKM(2,3;3|u)&u\IKvM(2,3;3|u) \\
\end{array}\right)\\={}& \det \begin{pmatrix}O(\log u) & \IKM(1,3;1)+O(u)&O(1/\sqrt{u}) \\
-V_{4}+O(\sqrt{u}\log u) & O(u)&O(1/\sqrt{u}) \\
O(u\log u)&O(u)&\frac{1}{4}+O(\sqrt{u}) \\
\end{pmatrix}\\={}&\frac{\pi^2V_4}{2^6}+o(1),\quad\text{as }u\to0^+,\end{split}
\end{align}thereby proving our claims.\end{proof}

\begin{proposition}[Factorization of $ \check \varOmega_3(1)$]We have the following factorization \begin{align}\check
\varOmega_3(1)=\frac{\IKM(1,2;1)}{2^{3}}\det\check {\mathbf M}_2
\end{align}where\begin{align}
\det\check{\mathbf M}_2:=\det
\begin{pmatrix}\IKM(0,5;1) & \IKM(0,5;3) \\
\IKM(2,3;1) & \IKM(2,3;3) \\
\end{pmatrix}=\frac{2 \pi^{3}}{15\sqrt{15}}m(1+x_1+x_2+x_3+x_4)
.\label{eq:detMv2}
\end{align}\end{proposition}

\begin{proof}Setting $u=1$ in \eqref{eq:IK_echelon}, and referring back to \eqref{eq:IK_reduction}, we may equate $ 2^3\check\varOmega_3(1)$ with \begin{align}\begin{split}&
\det\begin{pmatrix}\IKM(0,5;1)&\IKM(2,3;1)&\IKM(2,3;1) \\
\IKpM(0,5;1)&\IpKM(2,3;1)&\IKpM(2,3;1) \\
\IKM(0,5;3)&\IKM(2,3;3)&\IKM(2,3;3) \\
\end{pmatrix}\\={}&\det\begin{pmatrix}\IKM(0,5;1)&\IKM(2,3;1)&0 \\
\IKpM(0,5;1)&\IpKM(2,3;1)&-\IKM(1,2;1) \\
\IKM(0,5;3)&\IKM(2,3;3)&0 \\
\end{pmatrix}\\={}&\IKM(1,2;1)\det\check {\mathbf M}_{2}=\frac{\pi\det\check {\mathbf M}_2}{3\sqrt{3}}.\end{split}
\end{align}

Substituting into the integral representation for  $\check \varOmega_3(1)$ in \eqref{eq:check_Omega_3_1_int_repn}, we see that\begin{align}
\frac{45\sqrt{5}\pi\det\check {\mathbf M}_2}{2^3\sqrt{3}}=\pi^{2 }V_4+\frac{\pi^{4}}{4}\int_{0}^\infty \frac{1-J_0(t)}{t}[J_0(t)]^4\D t.
\end{align}Meanwhile, integrating by parts, we find\begin{align}\begin{split}
\int_0^\infty\frac{1-J_0(t)}{t}[J_0(t)]^4\D t={}&4\int_{0}^\infty J_1(t)[J_{0}(t)]^{3}\log t\D t-5\int_{0}^\infty J_1(t)[J_{0}(t)]^{4}\log t\D t\\={}&m(1+x_1+x_2+x_3+x_4)-m(1+x_1+x_2+x_3),\end{split}
\label{eq:Mahler_diff}\end{align}as a result of Broadhurst's integral representation for Mahler measures, given  in \eqref{eq:BroadhurstMahler}. Combining the last two equations while recalling  $ m(1+x_1+x_2+x_3)=\frac{4V_4}{\pi^2}$ from \eqref{eq:V3V4}, we achieve our goal.  \end{proof}


\subsection{Relation between $ \det \check {\mathbf N}_2$ and $ m(1+x_1+x_2+x_3+x_4+x_5)$}As a variation on the Wro\'nskian determinant \begin{align}
\omega_{4}(u)=W\left[ \frac{\IvKM(1,5;1|u)+5\IKvM(1,5;1|u)}{6},\IvKM(2,4;1|u),\IvKM(3,3;1|u) ,\IKvM(2,4;1|u)\right]
\end{align}treated in \S\ref{sec:detN2}, we consider its ``vacuum analog''\begin{align}
\check \omega_4(u)=W[\IKvM(0,6;1|u),\IvKM(2,4;1|u),\IvKM(3,3;1|u) ,\IKvM(2,4;1|u)].
\end{align}\begin{lemma}[Differential equation for $ \check\omega_4(u)$]For $ u\in(0,1)$, we have\begin{align}
\begin{split}&D^{1}\omega_4(u)\\={}&2\omega_4 (u)D^1\log\frac{1}{u^2 (1-u)(9-u) ( 25-u)}\\{}&+\frac{15\log u}{4u^2 (1-u)(9-u) ( 25-u)}\det\begin{pmatrix}D^0\nu^1_{2,2}(u) & D^0\nu^1_{2,3}(u) & D^0\nu^1_{2,4}(u) \\
D^1\nu^1_{2,2}(u) & D^1\nu^1_{2,3}(u) & D^1\nu^1_{2,4}(u) \\
D^2\nu^1_{2,2}(u) & D^2\nu^1_{2,3}(u) & D^2\nu^1_{2,4}(u) \\
\end{pmatrix},\end{split}
\end{align}where $\nu^1_{2,2}(u) =\IvKM(2,4;1|u)$, $ \nu^1_{2,3}(u)=\IvKM(3,3;1|u)$ and $\nu^1_{2,4}(u)=\IKvM(2,4;1|u) $.\end{lemma}\begin{proof}With the fourth-order Vanhove operator $\widetilde L_4$ defined in \eqref{eq:VL4}, we can establish (using methods similar to those in Lemma \ref{lm:Vanhove_L3_ODE}) the following differential equations: \begin{align}
\begin{cases}\widetilde L_4\IKvM(0,6;1|u)=\frac{15}{4}\log u, & \forall u\in(0,\infty); \\\widetilde L_4\IvKM(2,4;1|u)= 0,& \forall u\in(0,9);\\
\widetilde L_4\IvKM(3,3;1|u)= 0,& \forall u\in(0,1);\ \ \\
\widetilde L_4\IKvM(2,4;1|u)= 0, & \forall u\in(0,\infty). \\
\end{cases}\label{eq:PF_Vv6_u_all}
\end{align}

One can subsequently differentiate $\check \omega_4(u) $, with manipulations similar to those intended for $ \check\varOmega_3(u)$.  \end{proof}\begin{proposition}For $u\in(0,1)$, we have \begin{align}
\det\begin{pmatrix}D^0\nu^1_{2,2}(u) & D^0\nu^1_{2,3}(u) & D^0\nu^1_{2,4}(u) \\
D^1\nu^1_{2,2}(u) & D^1\nu^1_{2,3}(u) & D^1\nu^1_{2,4}(u) \\
D^2\nu^1_{2,2}(u) & D^2\nu^1_{2,3}(u) & D^2\nu^1_{2,4}(u) \\
\end{pmatrix}=\frac{\pi ^6}{80 u^2 (1-u) (9-u) (25-u)}\int_0^\infty J_0(\sqrt{u}t)[J_0(t)]^5t\D t.\label{eq:psi_p5}
\end{align}Consequently, there exists a constant $ \check c_4\in\mathbb R$ such that \begin{align}\begin{split}
[u^2 (1-u) (9-u) (25-u)]^{2}\check\omega_4(u)={}&\check c_4+\frac{3 \pi ^6 \sqrt{u} \log u}{32} \int_0^\infty J_1(\sqrt{u}t)[J_0(t)]^5\D t\\{}&-\frac{3 \pi ^6}{16}\int_{0}^\infty \frac{1-J_0(\sqrt{u}t)}{t}[J_0(t)]^5\D t\end{split}\label{eq:check_c4}
\end{align}is valid for $ u\in(0,1)$.\end{proposition}\begin{proof}First, we point out that  \begin{align}
\widetilde L_4\left[ u^2 (1-u) (9-u) (25-u) \det\begin{pmatrix}D^{0}f_{1}(u) & D^{0}f_2(u)&D^0f_3(u) \\
D^{1}f_{1}(u) & D^{1}f_2(u) &D^1f_3(u)\\
D^{2}f_{1}(u) & D^{2}f_2(u) &D^2f_3(u)\\
\end{pmatrix}\right]=0
\end{align} is true for any three functions $ f_1,f_2,f_{3}\in\ker\widetilde L_4$ residing the null space of $ \widetilde L_4$. So we may assert that there are constants $ C_1,C_2,C_3,C_4$ satisfying\begin{align}\begin{split}
\psi_3(u):={}&
u^2 (1-u) (9-u) (25-u) \det\begin{pmatrix}D^0\nu^1_{2,2}(u) & D^0\nu^1_{2,3}(u) & D^0\nu^1_{2,4}(u) \\
D^1\nu^1_{2,2}(u) & D^1\nu^1_{2,3}(u) & D^1\nu^1_{2,4}(u) \\
D^2\nu^1_{2,2}(u) & D^2\nu^1_{2,3}(u) & D^2\nu^1_{2,4}(u) \\
\end{pmatrix}\\={}& C_1\frac{\IvKM(1,5;1|u)+5\IKvM(1,5;1|u)}{6}+C_{2}\IvKM(2,4;1|u)\\{}&+C_3\IvKM(3,3;1|u)+C_{4}\IKvM(2,4;1|u),\end{split}
\end{align}for $u\in(0,1)$.

Next, we point out that the following limits\begin{align}
\lim_{u\to0^+}\psi_3(u)=\frac{3\pi^{2}}{8}\IKM(1,4;1)\quad\text{and}\quad\lim_{u\to0^+}D^1\psi_3(u)=\frac{3\pi^2}{32}\IKM(1,4;3).\label{eq:psi_3_0}
\end{align} allow us to determine\begin{align}
C_1=0,\quad C_2=\frac{3\pi^{2}}{8},\quad C_3=0,\quad C_4=0.
\end{align}

Here, before evaluating $ \lim_{u\to0^+}\psi_3(u)$, we  put down \begin{align}\begin{split}&
2^{3}u^2\det\begin{pmatrix}D^0\nu^1_{2,2}(u) & D^0\nu^1_{2,3}(u) & D^0\nu^1_{2,4}(u) \\
D^1\nu^1_{2,2}(u) & D^1\nu^1_{2,3}(u) & D^1\nu^1_{2,4}(u) \\
D^2\nu^1_{2,2}(u) & D^2\nu^1_{2,3}(u) & D^2\nu^1_{2,4}(u) \\
\end{pmatrix}\\={}&\det\left(\begin{array}{rrr}\nu^1_{2,2}(u) & \nu^1_{2,3}(u) & \nu^1_{2,4}(u)\\\sqrt{u}\acute\nu^1_{2,2}(u) & \sqrt{u}\acute\nu^1_{2,3}(u) & \sqrt{u}\acute\nu^1_{2,4}(u)\\\nu^2_{2,2}(u) & \nu^2_{2,3}(u) & \nu^2_{2,4}(u)\end{array}\right),\end{split}
\end{align} where the last determinant is asymptotic to   (cf.~Propositions \ref{prop:W_N2_0} and \ref{prop:W_M2_0})\begin{align}\begin{split}&
\det \begin{pmatrix}\mu^{1}_{2,1}+O(u) & \mu^{1}_{2,2}+O(u) &O(\log u) \\
O(u) &O(u) & -\mu^{1}_{2,2}+o(1) \\
\mu^{2}_{2,1}+O(u) & \mu^{2}_{2,2}+O(u)&O(\log u) \\
\end{pmatrix}\\={}&\IKM(2,3;1)\det \begin{pmatrix}\mu^{1}_{2,1} & \mu^{1}_{2,2}  \\
\mu^{2}_{2,1} & \mu^{2}_{2,2} \\
\end{pmatrix}+o(1)\\={}&\frac{2\pi^3\IKM(2,3;1)}{\sqrt{3^{3}5^5}}+o(1)\end{split}
\end{align} in the $ u\to0^+$ limit. Here, we recall from \cite[][Theorem 2.2.2 and Proposition 3.1.2]{Zhou2017WEF} that \begin{align} \IKM(2,3,1)=\frac{\sqrt{15}}{2\pi}\IKM(1,4,1),\end{align}so the evaluation of $ \lim_{u\to0^+}\psi_3(u)$ in \eqref{eq:psi_3_0} is now confirmed.

To compute  $ \lim_{u\to0^+}D^1\psi_3(u)$, we need the observations that $D^1[u^2D^2I_0(\sqrt{u}t)]=\frac{t^{3}\sqrt{u}}{8}I_1(\sqrt{u}t) $ and $ D^1[u^2$ $D^2K_0(\sqrt{u}t)]=-\frac{t^{3}\sqrt{u}}{8}K_1(\sqrt{u}t)$, which entail \begin{align}\begin{split}
D^1\psi_3(u)={}&\frac{(-3 u^2+70 u-259)\psi_3(u)}{(1-u) (9-u) (25-u)}+\frac{(1-u) (9-u) (25-u)}{16}\times\\{}&\times\det\begin{pmatrix}\nu^1_{2,2}(u) & \nu^1_{2,3}(u) & \nu^1_{2,4}(u)\\\acute\nu^1_{2,2}(u) &\acute\nu^1_{2,3}(u) & \acute\nu^1_{2,4}(u)\\\acute\nu^2_{2,2}(u) & \acute\nu^2_{2,3}(u) & \acute\nu^2_{2,4}(u)\end{pmatrix}.\end{split}
\end{align}Here, as $ u\to0^+$, the last determinant is asymptotic to\begin{align}
\det\begin{pmatrix}\mu^{1}_{2,1}+O(u) & \mu^{1}_{2,2}+O(u) &O(\log u)\\\frac{\sqrt{u}}{2}[\mu^{2}_{2,1}+O(u)] & \frac{\sqrt{u}}{2}[\mu^{2}_{2,2}+O(u)] & -\frac{\mu^{1}_{2,2}+o(1)}{\sqrt{u}}\\\frac{\sqrt{u}}{2}[\mu^{3}_{2,1}+O(u)] & \frac{\sqrt{u}}{2}[\mu^{3}_{2,2}+O(u)] & -\frac{\mu^{2}_{2,2}+o(1)}{\sqrt{u}}\end{pmatrix}.
\end{align}We recall  the following closed-form formulae (conjectured in \cite[][(95)--(100)]{BBBG2008}, proved in \cite[][\S3]{Zhou2017WEF})\begin{align}\left\{\begin{array}{r@{\,=\,}lr@{\,=\,}lr@{\,=\,}l}\frac{\mu^{1}_{2,1}}{\pi^{2}}&C,&\frac{\mu^{2}_{2,1}}{\pi^{2}}&\left( \frac{2}{15} \right)^{2}\left( 13C-\frac{1}{10C} \right),&\frac{\mu^{3}_{2,1}}{\pi^{2}}&\left( \frac{4}{15} \right)^{3}\left( 43C-\frac{19}{40C} \right)\\\frac{2\mu^{1}_{2,2}}{\sqrt{15}\pi}&C,&\frac{2\mu^{2}_{2,2}}{\sqrt{15}\pi}&\left( \frac{2}{15} \right)^{2}\left( 13C+\frac{1}{10C} \right),&\frac{2\mu^{3}_{2,2}}{\sqrt{15}\pi}&\left( \frac{4}{15} \right)^{3}\left( 43C+\frac{19}{40C} \right)\end{array}\right.
\end{align} where  $ C=\frac{1}{240 \sqrt{5}\pi^{2}}\Gamma \left(\frac{1}{15}\right) \Gamma \left(\frac{2}{15}\right) \Gamma \left(\frac{4}{15}\right) \Gamma \left(\frac{8}{15}\right)$ is the ``Bologna constant'' attributed to Broadhurst \cite{Broadhurst2007,BBBG2008} and Laporta \cite{Laporta2008}. It is then clear that \begin{align}
\lim_{u\to0^+}D^1\psi_3(u)=-\frac{259 \pi ^4 C}{600}+\frac{\pi ^4 (2720C^2-1)}{6000C}=\frac{3\pi^2}{32}\IKM(1,4;3),
\end{align}as claimed in \eqref{eq:psi_3_0}.

Now, to guarantee the finiteness of both $ \lim_{u\to0^+}\psi_3(u)$ and $ \lim_{u\to0^+}D^1\psi_3(u)$, we must have $ C_1=C_4=0$. Fitting $\psi_3(u)=C_{2}\IvKM(2,4;1|u)+C_3\IvKM(3,3;1|u) $ to  \eqref{eq:psi_3_0}, we obtain $ C_2=\frac{3\pi^2}{8},C_3=0$.

Last, but not the least, we recall from \cite[][Lemma 2.1]{Zhou2017PlanarWalks}  that \begin{align}
\frac{p_5(\sqrt{u})}{\sqrt{u}}:=\int_0^\infty J_0(\sqrt{u}t)[J_0(t)]^5t\D t=\frac{30\IvKM(2,4;1|u)}{\pi^{4}},\quad \forall u\in[0,1],
\end{align}which turns $ \psi_3(u)=\frac{3\pi^2}{8}\IvKM(2,4;1|u)$ into $ \psi_3(u)=\frac{\pi^{6}}{80}\frac{p_5(\sqrt{u})}{\sqrt{u}}$, just as stated in \eqref{eq:psi_p5}.

Following procedures similar to  those in Proposition \ref{prop:check_C3}, we can deduce \eqref{eq:check_c4} from  \eqref{eq:psi_p5}.\end{proof}In the next proposition, we study the determinant\begin{align}\begin{split}&
2^{6}u^{4}\check\omega_4(u)\\={}&\det\left(\begin{array}{rrrr}\IKvM(0,6;1|u)&\nu^1_{2,2}(u) & \nu^1_{2,3}(u) & \nu^1_{2,4}(u) \\
\sqrt{u}\IKpM(0,6;1|u)&\sqrt{u}\acute\nu^1_{2,2}(u) & \sqrt{u}\acute\nu^1_{2,3}(u) & \sqrt{u}\acute\nu^1_{2,4}(u) \\
\IKvM(0,6;3|u)&\nu^2_{2,2}(u) & \nu^2_{2,3}(u) & \nu^2_{2,4}(u) \\\sqrt{u}\IKpM(0,6;3|u)&\sqrt{u}\acute\nu^2_{2,2}(u) & \sqrt{u}\acute\nu^2_{2,3}(u) & \sqrt{u}\acute\nu^2_{2,4}(u)
\end{array}\right)\end{split}\label{eq:IK_echelon_1}
\end{align}
  in the $ u\to0^+$ limit.\begin{proposition}[Factorization of $ \check \omega_4(0^+)$]We have \begin{align}\begin{split}\check
c_{4}={}&3^45^4\lim_{u\to0^+}u^4\check \omega_4(u)=-\frac{45 \sqrt{15} \pi ^3}{32}\det\check{\mathbf M}_2\\={}&-\frac{3\pi^6}{16}m(1+x_1+x_2+x_3+x_4).\end{split}
\end{align} Consequently, we have \begin{align}\begin{split}&
2^{12}3^{2}\lim_{u\to1^-}(1-u)^2\check\omega_4(u)\\={}&-\frac{3\pi^6}{16}m(1+x_1+x_2+x_3+x_4)-\frac{3 \pi ^6}{16}\int_{0}^\infty \frac{1-J_0(t)}{t}[J_0(t)]^5\D t\\={}&-\frac{3\pi^6}{16}m(1+x_1+x_2+x_3+x_4+x_{5}).\label{eq:check_omega_4_1_int_repn}
\end{split}\end{align}\end{proposition}
\begin{proof}Using methods in Proposition \ref{prop:W_N2_0}, we can  show that \begin{align}\begin{split}&
2^{6}u^{4}\check\omega_4(u)\\={}&\det\begin{pmatrix}O(\log u) & \mu^{1}_{2,1}+O(u) & \mu^{1}_{2,2}+O(u) &O(\log u)\\
\sqrt{u}\IKpM(0,6;1|u) & O(u) &O(u)&\sqrt{u}\acute\nu^1_{2,4}(u) \\
O(\log u) & \mu^{2}_{2,1}+O(u) & \mu^{2}_{2,2}+O(u)&O(\log u) \\\sqrt{u}\IKpM(0,6;3|u) & O(u) &O(u)&\sqrt{u}\acute\nu^2_{2,4}(u) \\
\end{pmatrix}\\={}&-\det \begin{pmatrix}\mu^{1}_{2,1} & \mu^{1}_{2,2} \\
\mu^{2}_{2,1} & \mu^{2}_{2,2} \\
\end{pmatrix}\det\begin{pmatrix}\sqrt{u}\IKpM(0,6;1|u) & \sqrt{u}\acute\nu^1_{2,4}(u) \\
\sqrt{u}\IKpM(0,6;3|u) & \sqrt{u}\acute\nu^2_{2,4}(u) \\
\end{pmatrix}\\{}&+O(u^{2}\log^2u),\quad\text{as }u\to0^+,\end{split}
\end{align}and \begin{align}\begin{split}&
\det\begin{pmatrix}\sqrt{u}\IKpM(0,6;1|u) & \sqrt{u}\acute\nu^1_{2,4}(u) \\
\sqrt{u}\IKpM(0,6;3|u) & \sqrt{u}\acute\nu^2_{2,4}(u) \\
\end{pmatrix}\\={}&\det\begin{pmatrix}-\IKM(0,5;1)+o(1) & -\IKM(2,3;1)+o(1) \\
-\IKM(0,5;3)+o(1) & -\IKM(2,3;3)+o(1) \\
\end{pmatrix}\\={}&\det\check{\mathbf M}_2+o(1),\quad\text{as }u\to0^+.\end{split}
\end{align}The rest of our claims then follow from familiar arguments in \S\ref{subsec:red_V5}.\end{proof}

To wrap up this section, we reduce $ \check\omega_4(u), u\to1^-$ to $ \det\check{\mathbf N}_2$.

\begin{proposition}[Factorization of $ \check \omega_4(1^{-})$]We have the following factorization \begin{align}
\lim_{u\to1^-}(1-u)^2\check\omega_4(u)=-\frac{\pi^2}{2^{11}}\det\check {\mathbf N}_2
\end{align}so that \begin{align}
\det\check{\mathbf N}_2:=\det
\begin{pmatrix}\IKM(0,6;1) & \IKM(0,6;3) \\
\IKM(2,4;1) & \IKM(2,4;3) \\
\end{pmatrix}=\frac{\pi^{4}}{96}m(1+x_1+x_2+x_3+x_4+x_{5})
.\label{eq:detMv2}
\end{align}\end{proposition}\begin{proof}Akin to Proposition \ref{prop:W_N2_1}, we have \begin{align}\begin{split}&
2^{6}u^{2}(1-u)^{2}\check\omega_4(u)\\={}&\det \begin{pmatrix}\IKM(0,6;1)+\circ&\nu^1_{2,2}(1)+\circ&\circ&\nu^1_{2,4}(1)+\circ \\
\sharp\ & \acute\nu^1_{2,2}(1)+\circ & \circ & \acute \nu^1_{2,4}(1)+\circ \\
\IKM(0,6;3)+\circ  & \nu^2_{2,2}(1)+\circ & \circ&\nu^1_{2,4}(1)+\circ  \\
\sharp\ & \sharp & \frac{1}{2}+\circ & \sharp \\
\end{pmatrix}\\={}&-\frac{1}{2}\det\begin{pmatrix}\IKM(0,6;1)+\circ&\IKM(2,4;1)+\circ&\IKM(2,4;1)+\circ \\
\sharp\ & \IpKM(2,4;1|1)+\circ  &  \IKpM(2,4;1|1)+\circ \\
\IKM(0,6;3)+\circ  & \IKM(2,4;3)+\circ & \IKM(2,4;3)+\circ \\
\end{pmatrix}+o(1)\end{split}
\end{align}where a hash  (resp.~circle)  stands for a bounded (resp.~infinitesimal) quantity, as $u$ approaches $1$ from below. Using the fact that $ \IKpM(2,4;1|1)-\IpKM(2,4;1|1)=-\IKM(1,3;1)=-\frac{\pi^2}{2^{4}}$, we can compute\begin{align}\begin{split}&
2^{6}\lim_{u\to1^-}u^{2}(1-u)^{2}\check\omega_4(u)\\={}&-\frac{1}{2}\det\begin{pmatrix}\IKM(0,6;1)&\IKM(2,4;1)&0 \\
\IKpM(0,6;1|1)\ & \IpKM(2,4;1|1)  &  -\frac{\pi^2}{2^{4}} \\
\IKM(0,6;3)  & \IKM(2,4;3) & 0 \\
\end{pmatrix}\\={}&-\frac{\pi^2}{2^5}\det\check{\mathbf N}_2,\end{split}
\end{align}so our conclusion follows immediately.\end{proof}

\subsection*{Acknowledgments}I am grateful to Dr.\ David Broadhurst for his thought-inspiring conjectures about algebraic relations among Feynman integrals, and  for his communications on latest progress in the arithmetic studies of Bessel moments \cite{Broadhurst2017Paris,Broadhurst2017CIRM,Broadhurst2017Higgs,Broadhurst2017DESY}. I dedicate this work to his 70th birthday.


\end{document}